\documentclass[10pt,reqno]{amsart}

\usepackage[latin2]{inputenc}
\usepackage{amsmath}
\usepackage{graphicx}
\usepackage{amssymb}
\usepackage{color}
\usepackage{esint}
\usepackage{amsthm}
\usepackage{epsfig}
\usepackage[english]{babel}

\usepackage[left=3.3cm,right=3.3cm,bottom=2.7cm,top=3cm]{geometry}
\usepackage{nicefrac}

\newtheorem{theorem}{Theorem}

\newtheorem{lemma}[theorem]{Lemma}

\newtheorem*{theorem*}{Theorem}

\def\XXint#1#2#3{{\setbox0=\hbox{$#1{#2#3}{\int}$ }
\vcenter{\hbox{$#2#3$ }}\kern-.6\wd0}}

\newcommand{\supp}{\operatorname{supp}}

\newcommand{\ent}{\mathbf{H}}
\newcommand{\nrg}{\mathbf{E}}
\newcommand{\dss}{\mathbf{D}}
\newcommand{\aux}{\mathbf{G}}
\newcommand{\ival}{J^1_K}
\newcommand{\hval}{J^{\nicefrac12}_K}

\newcommand{\edge}{\rho}
\newcommand{\dff}{\mathrm D}
\newcommand{\mratio}{\Lambda}
\newcommand{\cell}{\delta}
\newcommand{\hi}{{\nicefrac12}}
\newcommand{\kph}{{k+\nicefrac12}}
\newcommand{\kmh}{{k-\nicefrac12}}
\newcommand{\kapph}{{\kappa+\nicefrac12}}
\newcommand{\kapmh}{{\kappa-\nicefrac12}}
\newcommand{\dsol}{z}
\newcommand{\dd}{\,\mathrm{d}}
\newcommand{\dn}{\mathrm{d}}
\newcommand{\setR}{{\mathbb R}}

\begin{document}

\title[Waiting time phenomena in discretizations]{The Waiting time phenomenon in spatially discretized porous~medium and thin~film equations}

\author{Julian Fischer}
\address{Julian Fischer \\ IST Austria \\ Am Campus 1 \\ 3400 Klosterneuburg \\ Austria}
\email{julian.fischer@ist.ac.at}
\author{Daniel Matthes}
\address{Daniel Matthes \\ TU Munich \\ Zentrum Mathematik \\ Boltzmannstr. 3 \\ 85747 Garching bei M\"unchen \\ Germany}
\email{matthes@ma.tum.de}
\thanks{This research was supported by the DFG Collaborative Research Center TRR 109, ``Discretization in Geometry and Dynamics".}

\begin{abstract}
  Various degenerate diffusion equations exhibit a waiting time phenomenon:
  Dependening on the ``flatness'' of the compactly supported initial datum at the boundary of the support,
  the support of the solution may not expand for a certain amount of time.  
  We show that this phenomenon is captured by particular Lagrangian discretizations
  of the porous medium and the thin-film equations,
  and we obtain suffcient criteria for the occurrence of waiting times
  that are consistent with the known ones for the original PDEs.
  Our proof is based on estimates on the fluid velocity in Lagrangian coordinates.
  Combining weighted entropy estimates with an iteration technique \`a la Stampacchia
  leads to upper bounds on free boundary propagation.
  Numerical simulations show that the phenomenon is already clearly visible for relatively coarse discretizations.
\end{abstract}

\maketitle

\section{Introduction}

\subsection{The evolution equations and waiting times}
In this paper, we prove the occurrence of the waiting time phenomenon
in appropriate spatial discretizations of two degenerate parabolic evolution equations in one space dimension:
The second order \emph{porous medium} (or \emph{slow diffusion}) equation with given exponent $m>1$,
\begin{align}
  \label{eq:PME1}
  \partial_ t u = \partial_{xx}(u^m),
\end{align}
and the fourth order \emph{thin film} (or \emph{lubrication}) equation with linear mobility,
\begin{align}
  \label{eq:TF1}
  \partial_t u = -\frac12\partial_x(u\,\partial_{xxx}u).
\end{align}
Both equations are known to admit non-negative global weak solutions
for the initial value problem with
\begin{align}
  \label{eq:ic}
  u(0;x) = \bar u(x),
\end{align}
for any initial datum $\bar u$ that is continuous, non-negative, and of compact support \cite{BernisFriedman,MMS,VazquezBook}.
The integral of $u$ is preserved under the evolution;
by homogeneity, it is no loss of generality to restrict attention to solutions of unit mass.

In the qualitative analysis of \eqref{eq:PME1} and \eqref{eq:TF1}, 
one of the key objects of interest is the growth of the support of the solution $u$ in time.
For the second order porous medium equation \eqref{eq:PME1},
it is easily seen from comparison principles
that an initially compactly supported solution is compactly supported at any later time as well,
and that the support cannot shrink,
see e.g. \cite{VazquezBook} for an overview on these and related results.
This comparison also provides rough lower and upper bounds on the speed at which the diameter of the support grows. 
A complementary approach via entropy methods provides estimates on the asymptotic proximity of general solutions 
to the compactly supported self-similar Barenblatt solutions \cite{PorousAsymptotic4,Dolbeault,OttoPME},
and thus gives a quantitative indication for the expected growth of the support at large times.

A little less is known about solutions to the fourth order equation \eqref{eq:TF1},
for which no comparison principle is available.
To prevent ill-posedness of the thin film equation \eqref{eq:TF1}, an additional boundary condition must be specified on the boundary of the support $\partial \supp u(\cdot,t)$; the typical choice (that we shall make here as well) is to restrict to solutions with zero contact angle,
which can be formally expressed as the condition $\partial_xu(\bar x)=0$ at any $\bar x$ at the edge of support.
In the existence theory of weak solutions, this condition is enforced via
certain entropy or energy dissipation estimates \cite{BernisFriedman,BerettaBertschDalPasso,OnAFourthOrder,ThinViscous,
CauchyGruen} (for a stronger solution theory, we refer to \cite{GiacomelliKnuepferOtto,GiacomelliKnuepfer,GiacomelliGnannKnuepferOtto,
Gnann,Gnann2,GnannPetrache}).
For the thin-film equation with zero contact angle, it has been shown that
the support grows with finite speed \cite{BernisFinite},
and that solutions are asymptotically close to the compactly supported self-similar one \cite{CarrilloToscani,MMS}.
We remark that more refined information is available for thin film equations with higher degeneracy,
see e.\,g.\ \cite{BernisFriedman} for a result on the non-shrinkage of the support or \cite{DeNittiFischer,FischerSupportPropagationThinFilm,
FischerUpperBoundsThinFilmWeakSlippage} for a complete characterization of the waiting time and support propagation behavior in certain regimes.

Here we focus on the occurrence of waiting times,
which is a subtle phenomenon showing that
despite the aforementioned results on the \emph{eventual} uniform growth of the solution's support,
one cannot expect expansion to happen immediately after initialization.
Instead, the edge of support only moves when the solution has gained a certain steepness there; if the initial profile is very flat near the boundary,
then it takes a certain ``waiting time'' until mass has been re-distributed on the support
before the necessary degree of steepness has been reached.
Criteria on the initial condition for the occurrence of the waiting time phenomenon
have been established by various authors since the 1980's,
see \cite{LowerBounds} for a brief historical review on waiting times for degenerate parabolic equations,
and particularly \cite{WaitingTime} for the first significant result on thin film equations.
A sufficient criterion for the phenomenon, applied to \eqref{eq:PME1} and \eqref{eq:TF1}, is this:
let $a$ be the left edge of $\bar u$'s support, then a waiting time occurs there if
\begin{align}
  \label{eq:suffPME}
  \limsup_{\ell\downarrow a}\left[(\ell-a)^{-\frac{3m-1}{m-1}}\int_a^\ell \bar u(x)^m\dd x\right] &< \infty,  \\
  \label{eq:suffTF}
  \limsup_{\ell\downarrow a}\left[(\ell-a)^{-(5+4\alpha)}\int_a^\ell \bar u(x)^{1+\alpha}\dd x\right] &< \infty
  \quad \text{for some $\alpha>0$},
\end{align}
respectively.
Criterion \eqref{eq:suffPME} is essentially sharp for \eqref{eq:PME1};
sharpness of \eqref{eq:suffTF} for \eqref{eq:TF1} has been partially shown
just recently by the first author \cite{FischerStrongAndVeryWeakSlippage}.

\subsection{Lagrangian picture}
The spatial discretizations of \eqref{eq:PME1} and \eqref{eq:TF1} considered in the following
are based on the Lagrangian description of the evolution.
Despite the fact that the motion of the edge of the solution's support
--- the object of central interest when studying waiting times ---
is very conveniently described in Lagrangian coordinates,
the Lagrangian approach has apparently not been used so far in the literature.
For passage to the Lagrangian picture,
we consider both \eqref{eq:PME1} and \eqref{eq:TF1} as non-linear transport equations,
\begin{align}
  \label{eq:transport}
  \partial_t u + \partial_x\big(u\,V(u)\big) = 0,
\end{align}
with respect to a velocity $V(u)$ that depends on $u$,
\begin{align*}
  V(u) = -\frac{m}{m-1}\partial_x(u^{m-1}),
  \quad\text{and}\quad
  V(u) = \frac12\partial_{xxx}u,
  \quad\text{respectively}.
\end{align*}
The evolution equation is now written in terms of the Lagrangian map $X:[0,T]\times[0,1]\to\setR$
that traces the characteristics of \eqref{eq:transport}, i.e.,
\begin{align*}
  \partial_t X = V(u)\circ X.
\end{align*}
Intuitively, $t\mapsto X(t;\xi)$ is the trajectory of a particle.
We normalize $X$ to ``mass coordinates'', i.e.,
for each $\xi\in[0,1]$, the amount of mass to the left of $X(t;\xi)$ equals $\xi$.
Consequently, $t\mapsto X(t;0)$ and $t\mapsto X(t;1)$ trace, respectively,
the left and right edges of $u$'s support.
One can easily express $u$ in terms of $X$ via the identity $u(t;X(t;\xi))\partial_\xi X(t;\xi)=1$,
and then rewrite \eqref{eq:PME1} and \eqref{eq:TF1}
in terms of $X$ and $Z(t;\xi):=\frac1{\partial_\xi X(t;\xi)} = u(t;X(t;\xi))$ alone:
\begin{align}
  \label{eq:PMEX}
  \partial_t X & = -\partial_{\xi}\big(Z^m\big), \\
  \label{eq:TFX}
  \partial_t X &= \partial_\xi\left(\frac12Z^3\,\partial_{\xi\xi}Z+\frac14Z^2\big(\partial_\xi Z\big)^2\right).
\end{align}
In order to translate the full initial value problems \eqref{eq:PME1}\&\eqref{eq:ic} and \eqref{eq:TF1}\&\eqref{eq:ic}
into reasonable initial-boundary-value problems in Lagrangian coordinates,
we actually consider \eqref{eq:PMEX} and \eqref{eq:TFX} as equations in terms of $Z$ alone,
bearing in mind that $\partial_tZ=-Z^2\partial_\xi\partial_tX$.
Thus \eqref{eq:PMEX} and \eqref{eq:TFX} are actually a second and a fourth order parabolic PDE for $Z$, respectively.
The natural boundary conditions are $Z(0)=u(t;X(t;0))=0$ and $Z(1)=u(t;X(t;1))=0$ for both equations,
expressing that $X(t,0)$ and $X(t,1)$ mark the left and the right edge of the support.
The other two boundary conditions for the thin-film equation \eqref{eq:TFX} are more difficult to formulate in Lagrangian terms:
the assumption of zero contact angle formally manifests itself as $[\partial_\xi(Z^2)](0)=[\partial_\xi(Z^2)](1)=0$,
which expresses a subtle regularity property of $Z$.
In the discretization below, we interprete this boundary condition as homogeneous Neumann.

\subsection{Discretization}
For discretization of \eqref{eq:PMEX} and \eqref{eq:TFX},
we use finite differences with respect to the mass coordinate $\xi\in[0,1]$,
i.e., we subdivide $[0,1]$ into $K$ sub-intervals $[\xi_{k-1},\xi_k]$;
the $\xi_k$ are fixed in time.
The Lagrangian map $X$ is discretized by a time-dependent sequence $x=(x_0,x_1,\ldots,x_K)$ of positions $x_k(t)\in\setR$,
where $x_k(t)$ serves as approximation of $X(t;\xi_k)$.
Thinking of $X$ as the piecewise linear interpolation of the $x_k$'s with respect to the $\xi_k$'s,
the associated density function on $\setR$ is piecewise constant on each of the intervals $(x_\kph,x_\kmh)$,
with respective density values $\dsol_\kappa := \frac{\xi_\kph-\xi_\kmh}{x_\kph-x_\kmh}$;
here $\kappa\in \{\frac{1}{2},\frac{3}{2},\ldots,K-\frac{1}{2}\}$ is a half-integer index.
Then, with the usual notations $\dff$ and $\Delta$ for first and second order difference quotients
--- see Section \ref{sct:prelim} for details ---
our discretizations are given by
\begin{align}
  \label{eq:devolPME}
 \frac{\dn}{\dd t} x &= -\dff(\dsol^m), \\
  \label{eq:devolTF}
  \frac{\dn}{\dd t} x &= \dff \Big[\frac12\dsol^3 \Delta\dsol + \frac14\dsol^2 \Delta (\dsol^2) \Big],
\end{align}
respectively.
Both are augmented with homogeneous Dirichlet boundary conditions,
\begin{align}
  \label{eq:dhomDiri}
  \dsol_{-\hi} = \dsol_{K+\hi} = 0,
\end{align}
and for \eqref{eq:devolTF},
we additionally ask for homogeneous Neumann boundary conditions in the form
\begin{align}
  \label{eq:dhomNeum}
  \dsol_{-\nicefrac{3}{2}} = \dsol_{K+\nicefrac{3}{2}} = 0. 
\end{align}
The discretizations \eqref{eq:devolPME} and \eqref{eq:devolTF} have appeared
at various places in the literature, see e.g. \cite{Calvez,Gosse,Naldi}.
A thorough analysis has been performed in \cite{OsbergerMatthes,OsbergerMatthes2},
where --- among other properties --- convergence of the approximate solutions in the continuous limit is shown.

We remark that the original motivation for choosing \eqref{eq:devolPME} and \eqref{eq:devolTF} in this particular way
lies beyond the formal similarity to \eqref{eq:PMEX} and \eqref{eq:TFX}.
Namely, the latter two are gradient flows for the functionals
\begin{align*}
  \ent(X) = \int_0^1\frac{Z^{m-1}}{m-1}\dd\xi
  \quad\text{and}\quad
  \nrg(X) = \int_0^1\big(\partial_\xi\sqrt Z\big)^2\dd\xi,
\end{align*}
with respect to the $L^2$-Hilbert structure on the space of Lagrangian maps $X:[0,1]\to\setR$.
This is not a coincidence, but reflects the fact that
the original evolution equations \eqref{eq:PME1} and \eqref{eq:TF1} are metric gradient flows
with respect to the $L^2$-Wasserstein distance, for the Renyi entropy and the Dirichlet energy, respectively,
see \cite{AGS,OttoPME,RigorousLubricationApproximation}.
The ordinary differential equations \eqref{eq:devolPME} and \eqref{eq:devolTF} inherit that gradient flow structure
in the sense that they constitute gradient flows on $\setR^{K+1}$
for potentials that are approximations of the Renyi entropy and the Dirichlet energy for spatially discrete densities.
That additional gradient flow structure has been the key ingredient
for the convergence proofs in \cite{OsbergerMatthes,OsbergerMatthes2}.
There are further structural elements preserved, like convexity properties of $\ent$ and $\nrg$;
on basis of that, it has been proven in \cite{Osberger} that the discrete solutions to \eqref{eq:devolTF}
replicate the self-similar long-time asymptotics of solutions to \eqref{eq:TF1} very precisely.

For the analysis at hand, the gradient flow structure as such is of minor importance.
What is significant is a side effect:
the discrete evolution equations \eqref{eq:devolPME} and \eqref{eq:devolTF}
admit a variant of the following dissipation estimates for \eqref{eq:PME1} and \eqref{eq:TF1}, respectively:
\begin{align}
  \label{eq:ZestPME}
  -\frac{\dn}{\dd t}\int_0^1 \frac{Z^m}{m-1}\dd\xi
  &= \int_0^1\big[\partial_\xi(Z^m)\big]^2\dd\xi, \\
  \label{eq:ZestTF}
  -\frac{\dn}{\dd t}\int_0^1 \frac{Z^\alpha}\alpha\dd\xi
  &\ge \frac{1+\alpha-2\alpha^2}{13+4\alpha+\alpha^2}\int_0^1\big[Z^{3+\alpha}(\partial_{\xi\xi}Z)^2+Z^{1+\alpha}(\partial_\xi Z)^4\big]\dd\xi.
\end{align}
These are easily obtained --- at least formally --- using integration by parts.
There exist weighted variants of these estimates,
which have a smooth weight function $\phi\ge0$ under the integral.
These weighted estimates are the key element for our analysis of the waiting time phenomenon.
The spatially discrete versions of those are given in Lemma \ref{lem:dissPME} and Lemma \ref{lem:dissTF}, respectively.
We remark that a discrete analogue of the entropy dissipation estimate \eqref{eq:ZestTF} has also been of central importance for numerical schemes for the thin film equation in Eulerian coordinates, see \cite{GruenRumpf,NumericsGruen}.

\subsection{Results}
The two main results of our paper are rigorous proofs for the occurrence of the waiting time phenomenon
for spatially discrete solutions to \eqref{eq:devolPME} and \eqref{eq:devolTF}, respectively.
The setup is that an initial datum $\bar u$ for \eqref{eq:ic} is given,
which is continuous, non-negative, and positive in the iterior of its compact support $[a,b]$.
For a given discretization of mass space by grid points $\xi_0$ to $\xi_K$,
the discrete equations \eqref{eq:devolPME} and \eqref{eq:devolTF} are then solved
with initial data $\bar x_0$ to $\bar x_K$ for $x$ that are consistent with the grid,
namely such that
\begin{align}
  \label{eq:consistent}
  \int_{\bar x_0}^{\bar x_k} \bar u(x)\dd x = \xi_k
  \quad\text{for $k=0,1,\ldots,K$}.
\end{align}
In both cases, the result is that if a certain quantity $\bar b$
--- a quotient of integrals that measures the steepness of $\bar u$ near $x=a$ ---
is finite,
then the left edge of the support $x_0(t)$ barely moves over a time horizon that is the larger the smaller $\bar b$ is;
that time is independent of the mesh.
Here ``barely moves'' means that $x_0(t)$ deviates from its initial value $\bar x_0=a$
at most by a positive power of the left-most mass cell size $\cell_\hi$.
\begin{theorem}
  \label{thm:dpmemain}
  There is a constant $C$ that only depends on $m$
  and the non-uniformity of the used reference mesh $(\xi_k)_{k=0}^K$ in mass space (i.\,e.\ $\Lambda$ in \eqref{eq:meshratio})
  such that the following is true
  for all spatially discrete approximations of solutions to \eqref{eq:PME1} via \eqref{eq:devolPME}
  that are consistently initialized in the sense \eqref{eq:consistent}.
  Provided that
  \begin{align}
    \label{eq:defbarbPME}
    \bar b:=\sup_{\ell\in(a,b)}\frac{\int_{\bar x_0}^\ell\bar u(x)^m\dd x}{\left(\int_{\bar x_0}^\ell\bar u(x)\dd x\right)^{\frac{3m-1}{m+1}}}<\infty,
  \end{align}
  then a waiting time occurs at the left edge of support:
  \begin{align*}
    |x_0(t) - \bar x_0|\le 2\cell_{\hi}^{\frac{m-1}{m+1}}\sqrt{\bar bt}
    \quad\text{for all}\quad
    t\le C\bar b^{-\frac{m+1}{m-1}}.
  \end{align*}
\end{theorem}
\noindent
It is readily checked that some $\bar u$ satisfying
\begin{align}
  \label{eq:almostpower}
  C^{-1}(x-a)^p\le\bar u(x)\le C(x-a)^p  
\end{align}
with some $C>0$ for all $x$ sufficiently close to $a$
meets \eqref{eq:defbarbPME} if and only if $p\ge\frac2{m-1}$.
The same power $p$ is critical in the criterion \eqref{eq:suffPME}.
\begin{theorem}  
  \label{thm:dtfmain}
  Assume that the mass mesh is equi-distant, $\xi_k=k/K$.
  For each positive $\alpha<\frac1{32}$, there are constants $C$ and $M$ 
  such that the following is true for all spatially discrete approximations of solutions to \eqref{eq:TF1} via \eqref{eq:devolTF}
  that are consistently initialized in the sense \eqref{eq:consistent}.
  Provided that
  \begin{align}
    \label{eq:defbarb}
    \bar b:=\sup_{\ell\in(a,b)}\frac{\int_{\bar x_0}^\ell\bar u(x)^{1+\alpha}\dd x}{\left(\int_{\bar x_0}^\ell\bar u(x)\dd x\right)^{1+\frac45\alpha}}<\infty,
  \end{align}
  then a waiting time occurs at the left edge of support:
  \begin{align}
    \label{eq:dpmemain}
    |x_0(t) - \bar x_0|\le
    M\cell^{\frac15}\big(\bar b^4t^{1+\alpha}\big)^{\frac1{5+\alpha}}
    \quad\text{for all}\quad
    t\le C\bar b^{-\frac5\alpha}.
  \end{align}
\end{theorem}
Similarly as above, for initial data $\bar u$ of the form \eqref{eq:almostpower},
condition \eqref{eq:defbarb} defines the same critical power $p=4$ as \eqref{eq:suffTF}.

To the best of our knowledge, our results are the first analytically rigorous ones
on the preservation of the waiting time phenomenon under spatial discretization.
We further emphasize that our calculations are apparently the first ones on the topic of waiting times
that have been carried out consistently in the Lagrangian picture, which seems very natural.
Although the estimates are formulated for the spatially discretized equations,
it is easily deduced how they carry over to \eqref{eq:PMEX} and \eqref{eq:TFX}, respectively.

\section{Preliminaries on the discretization}
\label{sct:prelim}
\subsection{Indices}
Let a natural number $K\ge2$ of discretization intervals be fixed.
Define the index sets
\[ k\in\ival:=\{0,1,2,\ldots,K\}, \qquad
\kappa\in\hval:=\left\{\nicefrac12,\nicefrac32,\nicefrac52,\ldots,K-\nicefrac12\right\}
\]
for integer and for non-integer half-values, respectively.
$\ival$ and $\hval$ are used to label points and intervals in between points, respectively.

\subsection{Mass space discretization}
Next, let a mesh $(\xi_k)_{k\in\ival}$ in ``mass space'' is given, i.e.,
\begin{align*}
  0=\xi_0<\xi_1<\cdots<\xi_K=1.
\end{align*}
Intuitively, the interval lengths
\begin{align*}
  \cell_\kappa := \xi_\kapph-\xi_\kapmh \quad\text{for $\kappa\in\hval$},
\end{align*}
are (time-independent) ``mass lumps'';
our convention is that $\cell_{-\hi}=\cell_{K+\hi}:=0$.
For notational convenience, we further introduce
\begin{align*}
  \cell_k := \frac{\cell_\kph+\cell_\kmh}2 \quad\text{for $k\in\ival$},
\end{align*}
so in particular $\cell_0=\xi_1/2$.
As usual, the mesh ratio $\mratio\ge1$ for $(\xi_k)_{k\in\ival}$ is defined as
\begin{align}
  \label{eq:meshratio}
  \mratio:=\max\left\{
  \max_{k=1,\ldots,K-1}\left(\frac{\cell_{\kph}}{\cell_\kmh}\right),
  \max_{k=1,\ldots,K-1}\left(\frac{\cell_{\kmh}}{\cell_\kph}\right)
  \right\}.
\end{align}
For the dual meshes, this implies that
\begin{align}
  \label{eq:meshestimate}
  \cell_k\le\frac{1+\mratio}2\min(\cell_\kph,\cell_\kmh) \quad k=1,\ldots,K-1.
\end{align}
Further, we introduce the finite sequence $(k^*_i)_{i=1}^{I}$ of indices $k^*_i$ as follows:
\begin{itemize}
\item $k^*_1=1$.
\item Given $k^*_{i-1}$ for some $i>1$:
  if $2\xi_{k^*_{i-1}}\ge1$, then set $I:=i$ and $k^*_I:=K$.
  Otherwise, define $k^*_{i}$ as the smallest index $k$ such that $\xi_k\ge2\xi_{k^*_{i-1}}$.
\end{itemize}
There is an accompanying increasing sequence $(\edge_i)_{i=1}^I$ of masses $\edge_i=\xi_{k^*_i}\in(0,1]$ for $i<I$, and $\edge_I=1$.
By construction and by \eqref{eq:meshratio},
we have that
\begin{align}
  \label{eq:edges}
  2\edge_{i-1}\le\edge_i<(2+\mratio)\edge_{i-1} \quad i=2,3,\ldots,I-1.
\end{align}
As usual, an \emph{equidistant mesh} is one in which all cells have the same size $\delta_\kappa\equiv\delta:=1/K$, i.e., $\xi_k=k/K$.
In that case, $\mratio=1$, and one has $k^*_i=2^{i-1}$ for $i=1,\ldots,I-1$, and accordingly $\edge_i=2^{i-1}/K$.

\subsection{Grid functions and difference operators}
By a \emph{grid function}, we mean a map $f:\hval\to\setR$.
Its canonical interpretation is that of a function on $[0,1]$
that is piecewise constant on the intervals $(\xi_\kapmh,\xi_\kapph)$ with respective values $f_\kappa$.
We define a difference operator $\dff$ for grid functions $f$ 
such that $\dff_kf$ is defined for $k=1,2,\ldots,I-1$ in the canonical way:
\begin{align*}
  \dff_kf = \frac{f_\kph-f_\kmh}{\cell_k}.
\end{align*}
We shall often assume additional values $f_{-\hi}$ and $f_{K+\hi}$,
such that $\dff_0f$ and $\dff_Kf$ are defined as well.
The difference operator is accompanied by a discete Laplacian $\Delta$,
which maps a grid function $f$
(augmented with boundary values $f_{-\hi}$ and $f_{K+\hi}=0$)
to a grid function $\Delta f$ as follows:
\begin{align*}
  \Delta_\kappa f = \frac{\dff_\kapph f-\dff_\kapmh f}{\cell_\kappa} .
\end{align*}
This is in accordance with the standard rule for summation-by-parts,
\begin{align}
  \label{eq:sbp}
  - \sum_{\kappa\in\hval} f_\kappa\Delta_\kappa g\,\cell_\kappa = \sum_{k\in\ival}\dff_k f\dff_k g\,\cell_k.
\end{align}
Note that on an equi-distant grid, where all cells have the same length $\cell$,
the definition of the Laplacian coincides with the well-known finite-difference quotient,
\begin{align*}
  \Delta_\kappa f = \frac{f_{\kappa+1}-2f_\kappa+f_{\kappa-1}}{\cell^2}.
\end{align*}
\begin{lemma}
  For two grid functions $f$ and $g$, the following product rule holds:
  \begin{align}
    \label{eq:productrule1}
    \dff_k(fg) = \frac12(f_\kph+f_\kmh)(\dff_kg) + \frac12(\dff_kf)(g_\kph+g_\kmh).
  \end{align}
  Moreover, if the grid is equi-distant, then
  \begin{align}
    \label{eq:productrule2}
    \Delta_\kappa(fg) = f_\kappa (\Delta_\kappa g)+(\Delta_\kappa f)g_\kappa
    + (\dff_\kapph f)(\dff_\kapph g) + (\dff_\kapmh f)(\dff_\kapmh g).
  \end{align}
\end{lemma}
A formula similar to \eqref{eq:productrule2} holds for non equi-distant meshes as well,
with non-trivial coefficients in front of the product of first derivatives.
\begin{proof}
  Both rules follow by straight-forward calculation.
  On the one hand,
  \begin{align*}
    \dff_k(fg)
    &= \frac{f_\kph g_\kph-f_\kmh g_\kmh}{\cell_k} \\
    &= \frac{(f_\kph+f_\kmh)(g_\kph-g_\kmh)+(f_\kph-f_\kmh)(g_\kph+g_\kmh)}{2\cell_k} \\
    &= \frac12(f_\kph+f_\kmh)\dff_kg + \frac12(\dff_kf)(g_\kph+g_\kmh).
  \end{align*}
  And on the other hand,
  \begin{align*}
    \Delta_\kappa(fg)
    &=\frac{f_{\kappa+1}g_{\kappa+1}-2f_\kappa g_\kappa + f_{\kappa-1}g_{\kappa-1}}{\cell^2} \\
    &= f_\kappa\frac{g_{\kappa+1}-2g_\kappa+g_{\kappa-1}}{\cell^2}
      + \frac{(f_{\kappa+1}-f_\kappa)g_{\kappa+1}+(f_{\kappa-1}-f_\kappa)g_{\kappa-1}}{\cell^2} \\
    &= f_\kappa(\Delta_\kappa g)
      + \frac{f_{\kappa+1}-2f_\kappa+f_{\kappa-1}}{\cell^2}g_\kappa
      + \frac{(f_{\kappa+1}-f_\kappa)(g_{\kappa+1}-g_\kappa)+(f_{\kappa-1}-f_\kappa)(g_{\kappa-1}-g_\kappa)}{\cell^2} \\
    &= f_\kappa (\Delta_\kappa g)+(\Delta_\kappa f)g
    + \dff_\kapph f\dff_\kapph g + \dff_\kapmh f\dff_\kapmh g.
  \end{align*}
\end{proof}

\subsection{Lagrangian map}
For solutions to \eqref{eq:devolPME} and \eqref{eq:devolTF},
the mass discretization $(\xi_k)_{k\in\ival}$ is fixed in time, 
while the corresponding discretization in physical space,
\begin{align*}
  -\infty<x_0(t)<x_1(t)<\cdots<x_K(t)<+\infty,
\end{align*}
evolves.
The vector $(x_k(t))_{k\in\ival}$ is the discretized analogue
of the time-dependent Lagrangian map $X(t):[0,1]\to\setR$,
which satisfies \eqref{eq:PMEX} or \eqref{eq:TFX}, respectively.
It is associated to a density function on $\setR$ of compact support,
which attains the constant value
\begin{align}
  \label{eq:defz}
  \dsol_\kappa(t) := \frac{\xi_\kapph-\xi_\kapmh}{x_\kapph(t)-x_\kapmh(t)}
\end{align}
in between the two consecutive points $\smash{x_\kapmh(t)}$ and $\smash{x_\kapph(t)}$.
In accordance with the boundary conditions \eqref{eq:dhomDiri} and \eqref{eq:dhomNeum},
we shall use $\dsol_\kappa(t)\equiv0$ for all half-integer indices $\kappa$ outside of $\smash{\hval}$.
For the conversion of the prescribed initial value $\bar u$ in \eqref{eq:ic}
to initial values $\bar x_k$ for the $x_k$,
we use the consistency relation \eqref{eq:consistent}.

\subsection{A discrete GNS inequality}
The following interpolation inequality plays an important role in the dissipation estimates that follow. We defer its elementary proof to the appendix.
\begin{lemma}
\label{LemmaGNS}
  For a grid function $f\ge0$ with $f_{-\hi}=0$ and any $r\in(0,1)$, $s\in(0,\frac13)$, 
  we have at each $k^*\in\{1,2,\ldots,K\}$ that
  \begin{align}
    \label{eq:dGNS}
    \sum_{\kappa=\hi}^{k^*-\hi}f_\kappa^2\cell_\kappa
    \le A_{2,r}\left(\sum_{k=0}^{k^*-1}(\dff_k f)^2\cell_k\right)^{\frac{1-r}{1+r}}\left(\sum_{\kappa=\hi}^{k^*-\hi}f_\kappa^{2r}\cell_\kappa\right)^{\frac2{1+r}}, \\
    \label{eq:dGNS4}
    \sum_{\kappa=\hi}^{k^*-\hi}f_\kappa^4\cell_\kappa
    \le A_{4,s}\left(\sum_{k=0}^{k^*-1}(\dff f)_k^4\cell_k\right)^{\frac{1-s}{1+3s}}\left(\sum_{\kappa=\hi}^{k^*-\hi}f_\kappa^{4s}\cell_\kappa\right)^{\frac4{1+3s}},
  \end{align}
  with the respective constants
  \begin{align*}
    A_{2,r} = \big[2^{1-2r}(1+r)^2(1+\mratio)\big]^{\frac{1-r}{1+r}},\quad
    A_{4,s} = \big[2^{1-12s}(1+3s)^4(1+\mratio)^3\big]^{\frac{1-s}{1+3s}}.
  \end{align*}
\end{lemma}
%


\section{The Discrete Porous Medium Equation}
In this section, we prove Theorem \ref{thm:dpmemain}.
We assume that some discretization in mass space via $(\xi_k)_{k\in\ival}$ is fixed.
And we consider the solution $x(t)=(x_k(t))_{k\in\ival}$
with associated densities $\dsol(t)=(\dsol_\kappa(t))_{\kappa\in\hval}$
to the discretized porous medium equation \eqref{eq:devolPME},
subject to the homogeneous Dirichlet conditions \eqref{eq:dhomDiri},
and for initial data $\bar x_k$ that are obtained from $\bar u$ via the consistency relation \eqref{eq:consistent}.
Using that
\begin{align*}
  \dot\dsol_\kappa
  = \frac{\dn}{\dd t}\left(\frac{\cell_\kappa}{x_\kapph-x_\kapmh}\right)
  = -\frac{\cell_\kappa}{(x_\kapph-x_\kapmh)^2}(\dot x_\kapph-\dot x_\kapmh)
  = -z_\kappa^2\dff_\kappa\dot x,
\end{align*}
we obtain the following equation for the densities $\dsol_\kappa$:
\begin{align}
  \label{eq:devolPMEz}
  \dot\dsol_\kappa = \dsol_\kappa^2\Delta_\kappa(\dsol^m).
\end{align}

\subsection{The dissipation estimate}
For notational simplicity, introduce the abbreviations
\begin{align*}
  \ent_i(\dsol) = \sum_{\kappa=\hi}^{k^*_i-\hi}\frac{\dsol_\kappa^{m-1}}{m-1}\cell_\kappa,
  \quad
  \dss_i(\dsol) =  \sum_{k=0}^{k^*_i-1}\big[\dff_k\big(\dsol^m\big)\big]^2\cell_k,
  \quad
  \aux_i(\dsol) = \sum_{\kappa=\hi}^{k^*_{i}-\hi}\dsol_\kappa^{2m}\cell_\kappa.
\end{align*}
The main goal of this subsection is to prove the following dissipation estimate.
\begin{lemma}
  \label{lem:dissPME}
  For each $i=1,2,\ldots,I-1$ and each $T>0$,
  \begin{equation}
    \label{eq:ddiss}
    \sup_{0\le t\le T}\ent_i\big(\dsol(t)\big)+\frac14\int_0^T\dss_i\big(\dsol(t)\big)\dd t
    \le
    8\rho_i^{-2}\int_0^T\aux_{i+1}\big(\dsol(t)\big)\dd t
    + \ent_{i+1}(\bar\dsol).
  \end{equation}
  The corresponding estimate in the case $i=I$ is
  \begin{align}
    \label{eq:ddissI}
    \sup_{0\le t\le T}\ent_I\big(\dsol(t)\big) +
    \int_0^T\dss_I\big(\dsol(t)\big)\dd t
    \le \ent_I(\bar\dsol).
  \end{align}
\end{lemma}
\noindent
Notice that \eqref{eq:ddissI} is our discretized version of the formal a priori estimate \eqref{eq:ZestPME}.
\begin{proof}
  The inequality \eqref{eq:ddissI} is easily derived:
  \begin{align*}
    - \frac{\dn}{\dd t}\ent_I(\dsol)
    = -\sum_\kappa \dsol_\kappa^{m-2}\dot\dsol_\kappa\cell_\kappa
    = -\sum_\kappa\dsol_\kappa^m\Delta_\kappa(\dsol^m)\cell_\kappa
    = \sum_k \big[\dff_k(\dsol^m)\big]^2\cell_k
    = \dss_I(\dsol),
  \end{align*}
  where we have used the summation by parts rule \eqref{eq:sbp}.
  Integrate this relation in time from $t=0$ to $t=T$ to obtain \eqref{eq:ddissI}.
  
  Now let $i\in\{1,\ldots,I-1\}$ be given.
  Define a monotonically non-increasing grid function $\phi$ with $0\le\phi\le1$ as follows:
  \begin{enumerate}
  \item $\phi_\kappa=1$ for each $\kappa\le k^*_i-\hi$,
  \item $\phi_\kmh=\frac{2\edge_i-\xi_k}{\edge_i}$ for $k=k^*_{i}+1,\ldots,k^*_{i+1}-1$ 
    --- this case does not occur if $k^*_{i+1}=k^*_i+1$,
  \item $\phi_\kappa=0$ for each $\kappa\ge k^*_{i+1}-\hi$.
  \end{enumerate}
  Notice that $\phi\ge0$ is guaranteed since $\xi_k<2\edge_i$ for all $k<k^*_{i+1}$ by definition of $k^*_{i+1}$.
  Further, one has
  \begin{align}
    \label{eq:phiestPME}
    0\le-\dff_k\phi \le \frac2{\rho_i}.
  \end{align}
  This is obvious for $k\le k^*_i-1$ or $k\ge k^*_{i+1}$, where $\dff_k\phi=0$,
  while 
  \begin{align*}
    0\le-\dff_k\phi = \frac{\xi_{k+1}-\xi_k}{\edge_i\cell_k} = \frac{\cell_\kph}{\edge_i\cell_k} < \frac2{\edge_i}
  \end{align*}
  for $k^*_i\le k\le k^*_{i+1}-1$,
  and finally,
  \begin{align*}
    -\dff_{k^*_{i+1}-1}\phi=\frac{2\edge_i-\xi_{k^*_{i+1}-1}}{\edge_i\cell_{k^*_{i+1}-1}}
    \le\frac{\cell_{k^*_{i+1}-\hi}}{\edge_i\cell_{k^*_{i+1}-1}}
    <\frac2{\edge_i},
  \end{align*}
  since $\xi_{k^*_{i+1}-1}<2\edge_i\le\xi_{k^*_{i+1}}=\xi_{k^*_{i+1}-1}+\cell_{k^*_{i+1}-\hi}$ by definition of $k^*_{i+1}$.

  After these preparations, we turn to estimate the dissipation of a weighted variant of $\ent$.
  Using the evolution equation \eqref{eq:devolPMEz},
  \begin{align*}
    J_\phi(t)
    :=-\frac{\dn}{\dd t} \sum_\kappa \frac{\dsol_\kappa^{m-1}}{m-1}\phi_\kappa^2\cell_\kappa
    &= -\sum_\kappa\dsol_\kappa^{m-2}\dot{\dsol}_\kappa\phi_\kappa^2\cell_\kappa \\
    &= -\sum_\kappa\phi_\kappa^2\dsol_\kappa^m[\Delta_\kappa(\dsol^m)]\cell_\kappa 
    = \sum_k [\dff_k(\phi^2\dsol^m)][\dff_k(\dsol^m)]\cell_k,
  \end{align*}
  where the last line follows from the summation-by-parts rule \eqref{eq:sbp}.
  With the product rule \eqref{eq:productrule1} applied twice,
  and with the aid of the Cauchy-Schwarz inequality,
  it follows that
  \begin{align*}
    J_\phi(t) 
    &= \frac12\sum_k [\dff_k(\dsol^m)]^2(\phi_\kph^2+\phi_\kmh^2)\cell_k
      +\frac14\sum_k[\dff_k(\dsol^m)](\phi_\kph+\phi_\kmh)\,(\dsol_\kph^m+\dsol_\kmh^m)\dff_k\phi\cell_k \\
    &\ge\frac12\sum_k [\dff_k(\dsol^m)]^2\Big[(\phi_\kph^2+\phi_\kmh^2)-\frac14(\phi_\kph+\phi_\kmh)^2\Big]\cell_k
      - \frac12\sum_k (\dsol_\kph^{m}+\dsol_\kmh^{m})^2[\dff_k\phi]^2\cell_k \\
    &\ge\frac14\sum_k [\dff_k(\dsol^m)]^2(\phi_\kph^2+\phi_\kmh^2)\cell_k
      - \sum_k (\dsol_\kph^{2m}+\dsol_\kmh^{2m})[\dff_k\phi]^2\cell_k .
  \end{align*}
  So far, the calculation is valid for an arbitrary grid function $\phi$.
  Now we use $\phi$'s defining properties, and \eqref{eq:phiestPME}:
  \begin{align}
    \label{eq:dummy101}
    J_\phi(t) 
    \ge \frac14\sum_{k=0}^{k^*_i-1}[\dff_k(\dsol^m)]^2\cell_k
    - 8\edge_i^{-2}\sum_{\kappa=\hi}^{k^*_{i+1}-\hi}\dsol_\kappa^{2m}\cell_\kappa
    = \frac14\dss_i(\dsol(t)) - 8\edge_i^{-2}\aux_{i+1} (\dsol(t)).
  \end{align}
  On the other hand, we have for each $t^*\in[0,T]$ that:
  \begin{align*}
    \int_0^{t^*}J_\phi(t)\dd t
    &= \sum_\kappa \frac{\bar\dsol_\kappa^{m-1}}{m-1}\phi_\kappa^2\cell_\kappa
      - \sum_\kappa \frac{\dsol_\kappa(t^*)^{m-1}}{m-1}\phi_\kappa^2\cell_\kappa \\
    &\le \sum_{\kappa=1/2}^{k^*_{i+1}-\hi}\frac{\bar\dsol_\kappa^{m-1}}{m-1}\cell_\kappa
      - \sum_{\kappa=1/2}^{k^*_{i}-\hi}\frac{\dsol_\kappa(t^*)^{m-1}}{m-1}\cell_\kappa
      = \ent_{i+1}(\bar\dsol)-\ent_i(\dsol(t^*)).
  \end{align*}
  Integration of \eqref{eq:dummy101} with respect to time
  and taking the supremum over $t^*\in[0,T]$ yields \eqref{eq:ddiss}.
\end{proof}

\subsection{The Stampacchia iteration}
For $T>0$ to be determined below, introduce
\begin{align}
  \label{eq:defa}
  a_i&:=\edge_i^{-\frac{5m+1}{m+1}}\int_0^T\aux_{i}(\dsol(t))\dd t \quad \text{for $i=1,2,\ldots,I$}, \\
  \label{eq:defb}
  b&:=\max_{i=1,\ldots,I} \left(\edge_i^{-\frac{3m-1}{m+1}}\ent_i(\bar\dsol)\right).
\end{align}
Below, for $T$ as chosen in \eqref{eq:T4PME} we derive the inequalities 
\begin{align}
  \label{eq:stampacchiaPME}
  a_I\le cb^{1+\theta},\qquad a_i\le c[a_{i+1}+b]^{1+\theta} \quad \text{for each $i=1,2,\ldots,I-1$},
\end{align}
with $\theta>0$, and with $c>0$ so small that
\begin{align}
  \label{eq:siPME}
  (2b)^\theta c\le\frac12.
\end{align}
It then follows by an easy induction argument that $a_i\le b$ for all $i=1,2,\ldots,I$.
Indeed,
\begin{align*}
  a_I \le (cb^\theta) b \le 2^{-(\theta+1)}b \le b,
\end{align*}
and if $a_{i+1}\le b$, then also
\begin{align*}
  a_i \le \big(c[a_{i+1}+b]^\theta\big)[a_{i+1}+b] \le \big((2b)^\theta c\big)(2b) \le b.
\end{align*}
In particular, 
\begin{align}
  \label{eq:crucial}
  a_1=\edge_1^{-\frac{5m+1}{m+1}}\int_0^T\dsol_{\hi}(t)^{2m}\cell_{\hi}\dd t \le b,
\end{align}
which is the key estimate to conclude the proof of Theorem \ref{thm:dpmemain} in the next section.

The rest of this section is devoted to the derivation of \eqref{eq:stampacchiaPME} with \eqref{eq:siPME}.
Applying inequality \eqref{eq:dGNS} with $f:=\dsol(t)^m$, with $r:=\frac{m-1}{2m}$, and with $k^*:=k^*_{i}$ 
yields
\begin{align*}
  \sum_{\kappa=\hi}^{k^*_{i}-\hi}\dsol_\kappa(t)^{2m}\cell_\kappa
  \le A_{2,r}
  \left(\sum_{\kappa=\hi}^{k^*_{i}-\hi}\dsol_\kappa(t)^{m-1}\cell_\kappa\right)^{\frac{4m}{3m-1}}
  \left(\sum_{k=0}^{k_{i}^*-1}\big[\dff_k\big(\dsol(t)^m\big)\big]^2\cell_k\right)^{\frac{m+1}{3m-1}}
\end{align*}
that is,
\begin{align*}
  \aux_i(\dsol(t)) \le B\ent_i(\dsol(t))^{\frac{4m}{3m-1}} \left[\frac14\dss_i(\dsol(t))\right]^{\frac{m+1}{3m-1}},
\end{align*}
with $B:=4^{\frac{m+1}{3m-1}}(m-1)^{\frac{4m}{3m-1}}A_{2,r}$.
Now we integrate in time, use H\"older's inequality, and 
then invoke the dissipation estimate \eqref{eq:ddiss},
obtaining
\begin{align*}
  \int_0^T\aux_i(\dsol(t))\dd t
  &\le B \sup_{0\le t\le T}\ent_i(\dsol(t))^{\frac{4m}{3m-1}}\int_0^T\left[\frac14\dss(\dsol(t))\right]^{\frac{m+1}{3m-1}}\dd t \\
  &\le B T^{\frac{2(m-1)}{3m-1}}
    \left[\sup_{0\le t\le T}\ent_i(\dsol(t))\right]^{\frac{4m}{3m-1}}
    \left[\frac14\int_0^T\dss(\dsol(t))\dd t\right]^{\frac{m+1}{3m-1}} \\
  &\le B T^{\frac{2(m-1)}{3m-1}}\left[8\edge_i^{-2}\int_0^T\aux_{i+1}(\dsol(t))\dd t + \ent_{i+1}(\bar\dsol)\right]^{\frac{5m+1}{3m-1}}.
\end{align*}
In terms of $a_i$ and $b$ introduced in \eqref{eq:defa} and \eqref{eq:defb},
we obtain, recalling \eqref{eq:edges},
for $i=0,1,\ldots,I-1$ that
\begin{align*}
  a_i
  &\le B T^{\frac{2(m-1)}{3m-1}}\edge_i^{-\frac{5m+1}{m+1}}
    \left[8\edge_i^{-2}\int_0^T\aux_{i+1}(\dsol(t))\dd t + \ent_{i+1}(\bar\dsol)\right]^{\frac{5m+1}{3m-1}} \\
  &= B T^{\frac{2(m-1)}{3m-1}}
    \left[8\left(\frac{\edge_{i+1}}{\edge_i}\right)^{\frac{5m+1}{m+1}}\edge_{i+1}^{-\frac{5m+1}{m+1}}\int_0^T\aux_{i+1}(\dsol(t))\dd t
    + \left(\frac{\edge_{i+1}}{\edge_i}\right)^{\frac{3m-1}{m+1}}\edge_{i+1}^{-\frac{3m-1}{m+1}}\ent_{i+1}(\bar\dsol)\right]^{\frac{5m+1}{3m-1}} \\
  &\le CT^{\frac{2(m-1)}{3m-1}}[a_{i+1}+b]^{1+\frac{2(m+1)}{3m-1}}.
\end{align*}
with $C:=512(2+\mratio)^{15}B$.
For $i=I$, and with \eqref{eq:ddissI} instead of \eqref{eq:ddiss}, we obtain 
\begin{align*}
  a_I \le  CT^{\frac{2(m-1)}{3m-1}}b^{1+\frac{2(m+1)}{3m-1}}.
\end{align*}
The choice
\begin{align}
  \label{eq:T4PME}
  T:=2^{-\frac{5m+1}{2(m-1)}}C^{-\frac{3m-1}{2(m-1)}}b^{-\frac{m+1}{m-1}},
\end{align}
produces the family of inequalities in \eqref{eq:stampacchiaPME},
with $\theta=\frac{2(m+1)}{3m-1}>0$, and with
\begin{align*}
  c=CT^{\frac{2(m-1)}{3m-1}}b^{\frac{2(m+1)}{3m-1}},
\end{align*}
which satisfies \eqref{eq:siPME}, thanks to the choice of $T$ in \eqref{eq:T4PME}.

\subsection{End of the proof of Theorem \ref{thm:dpmemain}}
According to \eqref{eq:devolPME} and the Dirichlet boundary condition \eqref{eq:dhomDiri},
the position $x_0(t)$ of the left edge of the support of $\dsol$ satisfies
\begin{align*}
  \dot x_0 = -\dff_0(\dsol^m) = -\frac{\dsol_\hi^m-0}{\cell_0} = -2\xi_1^{-1}\dsol_\hi^m.
\end{align*}
Recall the choice of $T$ in \eqref{eq:T4PME}.
From \eqref{eq:crucial}, it follows that
\begin{align*}
  \int_0^T\dsol_\hi(t)^{2m}\dd t \le \xi_1^{\frac{4m}{m+1}}b.
\end{align*}
Combining this with the evolution equation for $x_0$, we obtain at time $t^*\in[0,T]$:
\begin{align*}
  |x_0(t^*)-a|^2 
  \le \left(\int_0^{t^*}|\dot x_0(t)|\dd t\right)^2
  \le t^*\int_0^T\dot x_0(t)^2\dd t
  \le 4\xi_1^{-2}t^*\int_0^{T}\dsol_{\hi}(t)^{2m}\dd t
  \le 4\xi_1^{2\frac{m-1}{m+1}}bt^*.
\end{align*}
We have thus verified the claim of Theorem \ref{thm:dpmemain},
provided we can also show that $b$ in \eqref{eq:defb} is estimated by $\bar b$ in \eqref{eq:defbarb}.
This is a consequence of the initially consistent discretization, see \eqref{eq:consistent}.
Indeed, by Jensen's inequality,
\begin{align*}
  \bar\dsol_\kappa^m
  = \left(\frac{\xi_\kapph-\xi_\kapmh}{\bar x_\kph-\bar x_\kmh}\right)^m
  = \left(\frac{\int_{\bar x_\kmh}^{\bar x_\kph}\bar u(x)\dd x}{\bar x_\kph-\bar x_\kmh}\right)^m
  \le \frac1{\bar x_\kph-\bar x_\kmh}\int_{\bar x_\kmh}^{\bar x_\kph}\bar u(x)^m\dd x,
\end{align*}
and therefore,
\begin{align*}
  \sum_{\kappa=\hi}^{k^*_i-\hi}\bar\dsol_\kappa^{m-1}\cell_\kappa
  = \sum_{\kappa=\hi}^{k^*_i-\hi} \bar\dsol_\kappa^m(\bar x_\kph-\bar x_\kmh)
  \le \sum_{\kappa=\hi}^{k^*_i-\hi}\int_{\bar x_\kmh}^{\bar x_\kph}\bar u(x)^m\dd x
  = \int_a^{\bar x_{k^*_i}}\bar u(x)^m\dd x.
\end{align*}
Combining this with the fact that
\begin{align*}
  \edge_i = \int_a^{\bar x_{k^*_i}}\bar u(x)\dd x,
\end{align*}
yields $b\le\bar b$.


\section{The thin-film equation}
This section is devoted to the proof of Theorem \ref{thm:dtfmain}.
Hence, we asssume an equi-distant mesh with $\xi_k=k/K$ and identical cell lengths $\cell=1/K$;
it follows in particular that $\rho_i=2^i/K$.
We consider the solution $x(t)$ with corresponding densities $\dsol(t)$ to \eqref{eq:devolTF},
subject to the homogeneous Dirichlet \eqref{eq:dhomDiri} and Neumann \eqref{eq:dhomNeum} boundary conditions,
and for initial data that are obtained from $\bar u$ by means of the consistency relation \eqref{eq:consistent}.
For $\kappa\in\hval$, the equation \eqref{eq:devolTF} entails:
\begin{align}
  \label{EvolutionEquationZ}
  \dot\dsol_\kappa
  =-\dsol_\kappa^2 \Delta_\kappa\Big[\frac12\dsol^3 \Delta\dsol + \frac14\dsol^2 \Delta [\dsol^2] \Big].
\end{align}

\subsection{The dissipation estimate}
For the dissipation estimate, we assume that some sufficiently small $\alpha>0$ is fixed.
The roles of $\ent$, $\dss$ and $\aux$ are now played by:
\begin{align*}
  \ent_i(\dsol) &= \sum_{\kappa=\hi}^{k_i^*-\hi} \frac{\dsol_\kappa^\alpha}\alpha\cell, \\
  \dss_i(\dsol) &= \sum_{\kappa=\hi}^{k_i^*-\hi}\left[
                  \dsol_\kappa^{3+\alpha}\big(\Delta_\kappa\dsol\big)^2
                  +\frac12\dsol_\kappa^{1+\alpha}\big([\dff _{\kappa-\hi}\dsol]^4+[\dff _{\kappa+\hi}\dsol]^4\big)
                  \right]\cell, \\
  \aux_i(\dsol) &= \sum_{\kappa=\hi}^{k_i^*-\hi}\dsol_\kappa^{5+\alpha}\cell.
\end{align*}
\begin{lemma}
  \label{lem:dissTF}
  Fix some $\alpha\in(0,\frac1{32})$.
  There are constants $c>0$ and $B,C\ge 1$ such that for each $\sigma\in(0,1)$,
  the following is true:
  for each index $i=2,3,\ldots,I-1$ and each time $T>0$,
  \begin{align}
    \label{eq:tfdiss}
    \begin{split}
      &\sup_{t\in[0,T]}\ent_i\big(\dsol(t)\big) + c\int_0^T\dss_i\big(\dsol(t)\big)\dd t \\
      &\le  B\sigma^{-4}\rho_{i+1}^{-4}\int_0^T\aux_{i+1}\big(\dsol(t)\big)\dd t
      + C\sigma\int_0^T\dss_{i+1}\big(\dsol(t)\big)\dd t
      + \ent_{i+1}(\bar\dsol).      
    \end{split}
  \end{align}
  For $i=I$, one has instead:
  \begin{align}
    \label{eq:tfdissI}
    \sup_{t\in[0,T]}\ent_I\big(\dsol(t)\big) + c\int_0^T\dss_I\big(\dsol(t)\big)\dd t
    \le \ent_{I}(\bar\dsol).    
  \end{align}
\end{lemma}

\subsubsection{Proof of \eqref{eq:tfdiss} --- preparation}
Throughout the proof, let some $\Phi\in C^2(\setR_{\ge0})$ be fixed with the properties
that $\Phi(\xi)=1$ for $\xi\le\frac12$, and $\Phi(\xi)=0$ for $\xi\ge\frac34$.
The constants $c$, $B$ and $C$ appearing in \eqref{eq:tfdiss} and \eqref{eq:tfdissI}
are expressible in terms of norms of $\Phi$ alone.
Given an index $i$ with $2\le i< I$, we define a grid function $\phi$ by
\begin{align*}
  \phi_\kappa = \Phi\left(\frac{\xi_\kapph}{\rho_{i+1}}\right).
\end{align*}
The properties of $\Phi$ entail
that $\phi_\kappa=1$ for $\kappa\le k_i^*-\hi$, and $\phi_\kappa=0$ for $\kappa\ge k_{i+1}^*-\hi$.
Moreover,
\begin{align}
  \label{eq:phisupp}
  \dff_k\phi = 0 \quad\text{for all $k>k^*_{i+1}-1$},
  \qquad
  \Delta_\kappa\phi = 0 \quad \text{for all $\kappa>k_{i+1}^*-\hi$},
\end{align}
and
\begin{align}
  \label{eq:phiest}
  \max_k|\dff_k\phi| \le \|\Phi\|_{C^1}\rho_{i+1}^{-1},
  \quad
  \max_\kappa|\Delta_\kappa\phi| \le \|\Phi\|_{C^2}\rho_{i+1}^{-2}.
\end{align}
Next, define the grid function $F$ by
\begin{align*}
  F_\kappa = \frac12\dsol_\kappa \Delta_\kappa\dsol + \frac14 \Delta_\kappa(\dsol^2)
  \stackrel{\eqref{eq:productrule2}}{=} \dsol_\kappa\Delta_\kappa\dsol + \frac14\big[(\dff_{\kapph}\dsol)^2+(\dff_{\kapmh}\dsol)^2\big],
\end{align*}
so that \eqref{EvolutionEquationZ} can be written as
\begin{align}
  \label{eq:devolTFz}
  \dot\dsol_\kappa = -\dsol_\kappa^2\Delta_\kappa[\dsol^2F].
\end{align}
For later reference, note that
\begin{align}
  \label{eq:FbyDiss}
  F_\kappa^2 \le 2\dsol_\kappa^2(\Delta_\kappa\dsol)^2
  + \frac14\big[(\dff_{\kapph}\dsol)^4+(\dff_{\kapmh}\dsol)^4\big].
\end{align}

\subsubsection{Proof of \eqref{eq:tfdiss} --- calculating the dissipation}
Equation \eqref{eq:devolTFz} and a summation by parts yield
\begin{align}
  J:=-\frac{\dn}{\dd t}\sum_\kappa\frac{\dsol_\kappa^\alpha}{\alpha}\phi_\kappa\cell
  = \sum_\kappa \dsol_\kappa^{\alpha+1}\phi_\kappa\Delta[\dsol^2F]_\kappa\cell 
  = \sum_\kappa \dsol_\kappa^2F_\kappa \Delta_\kappa[\dsol^{\alpha+1}\phi]\cell,
\end{align}
where we have used the Dirichlet boundary conditions \eqref{eq:dhomDiri}.
By the product rule \eqref{eq:productrule2}, we obtain
\begin{align}
  \label{eq:devilssum}
  \Delta_\kappa(\dsol^{\alpha+1}\phi)
  = \phi_\kappa\Delta_\kappa(\dsol^{\alpha+1})
  +(\Delta_\kappa\phi)\dsol_\kappa^{\alpha+1}
  +(\dff_{\kapph}\phi)\dff_\kapph(\dsol^{\alpha+1})
  +(\dff_{\kapmh}\phi)\dff_\kapmh(\dsol^{\alpha+1}),
\end{align}
and we write accordingly
\begin{align*}
  \sum_\kappa \dsol_\kappa^2F_\kappa \Delta_\kappa\big(\dsol^{\alpha+1}\phi\big)\cell
  = S_1 + S_2 + S_3,
\end{align*}
with
\begin{align*}
  S_1&:=\sum_\kappa \phi_\kappa\Delta_\kappa(\dsol^{1+\alpha})\dsol_\kappa^2F_\kappa\cell, \\
  S_2&:=\sum_\kappa(\Delta_\kappa\phi)\dsol_\kappa^{\alpha+3}F_\kappa\cell,\\
  S_3&:=\sum_\kappa \left[
       (\dff_{\kapph}\phi)\big(\dff_\kapph(\dsol^{\alpha+1})\big)
       +(\dff_{\kapmh}\phi)\big(\dff_\kapmh(\dsol^{\alpha+1})\big)
       \right]\dsol_\kappa^2F_\kappa\cell.
\end{align*}
We estimate each of the sums $S_1$ to $S_3$ from below.
Concerning $S_1$, we observe that
in view of the elementary estimate \eqref{eq:laplace1} from the appendix
--- applied with $f=\dsol$ and $p=\alpha$ ---
\begin{align*}
  \Delta_\kappa(\dsol^{\alpha+1})
  = (1+\alpha)\dsol_\kappa^\alpha\Delta_\kappa\dsol + R_\kappa,
\end{align*}
with a remainder term $R_\kappa$ that can be estimated as follows:
\begin{align*}
  |R_\kappa|\le \alpha\dsol_\kappa^{-1+\alpha}\big[(\dff_\kapph\dsol)^2+(\dff_\kapmh\dsol)^2\big],
\end{align*}
and therefore,
\begin{align*}
  S_1&\ge (1+\alpha)\sum_\kappa\phi_\kappa\dsol_\kappa^{2+\alpha}(\Delta_\kappa\dsol) F_\kappa\cell
       - \sum_\kappa\phi_\kappa\dsol_\kappa^{2}|R_\kappa||F_\kappa|\cell \\
     &\ge (1+\alpha)\sum_\kappa\phi_\kappa\Big(
       \dsol_\kappa^{3+\alpha}(\Delta_\kappa\dsol)^2
       +\frac14\dsol_\kappa^{2+\alpha}(\Delta_\kappa\dsol)\big[(\dff_\kapph\dsol)^2+(\dff_\kapmh\dsol)^2\big]
       \Big)\cell \\
     &\qquad -\alpha\sum_\kappa\phi_\kappa\dsol_\kappa^{1+\alpha}|F_\kappa|
       \big[(\dff_\kapph\dsol)^2+(\dff_\kapmh\dsol)^2\big]\cell\\
     &\ge (1+\alpha) \sum_\kappa\phi_\kappa\Big(
       \dsol_\kappa^{3+\alpha}(\Delta_\kappa\dsol)^2 
       -\frac\alpha4 \dsol_\kappa^{1+\alpha}\big[(\dff_\kapph\dsol)^2+(\dff_\kapmh\dsol)^2\big]^2 \\
     &\qquad\qquad\qquad\qquad\qquad -\frac{1+4\alpha}4\dsol_\kappa^{2+\alpha}|\Delta_\kappa\dsol|
       \big[(\dff_\kapph\dsol)^2+(\dff_\kapmh\dsol)^2\big]
       \Big)\cell.
\end{align*}
To estimate $S_2$, we use Young's inequality,
and recall \eqref{eq:phisupp}, \eqref{eq:phiest}, and \eqref{eq:FbyDiss},
to obtain
\begin{align*}
  |S_2|&\le \frac\sigma2\sum_{\kappa=\hi}^{k^*_{i+1}-\hi}\dsol_\kappa^{\alpha+1}F_\kappa^2\cell
         + \frac1{2\sigma}\sup_\kappa|\Delta_\kappa\phi|^2\sum_{\kappa=\hi}^{k^*_{i+1}-\hi}\dsol_\kappa^{\alpha+5}\cell\\
       &\le \sigma\dss_{i+1}(\dsol) + \frac{\|\Phi\|_{C^2}^2}{2\sigma}\rho_{i+1}^{-4}\aux_{i+1}(\dsol).    
\end{align*}
Finally, to estimate $S_3$, we first observe that
the elementary estimate \eqref{eq:moreelementary} from the appendix
implies that $|\dff_k(\dsol^{1+\alpha})|\le(\dsol_{\kph}^\alpha+\dsol_\kmh^\alpha)|\dff_k\dsol|$,
and then apply Young's inequality to the triple products, with exponents $4$, $4$ and $2$, respectively:
\begin{align*}
  |S_3|
  &\le \sum_\kappa\left[
    |\dff_{\kapph}\phi|(\dsol_{\kappa+1}^\alpha+\dsol_\kappa^\alpha)|\dff_\kapph\dsol|
    +|\dff_{\kapmh}\phi|(\dsol_{\kappa-1}^\alpha+\dsol_\kappa^\alpha)|\dff_\kapmh\dsol|
    \right]\dsol_\kappa^2F_\kappa\cell \\
  &\le \sigma\sum_{\kappa=\hi}^{k^*_{i+1}-\hi}\dsol_\kappa^{\alpha+1}F_\kappa^2\cell
    +\frac\sigma4\sum_{\kappa=\hi}^{k^*_{i+1}-\hi}\dsol_\kappa^{\alpha+1}
    \big[(\dff_\kapph\dsol)^4+(\dff_\kapmh\dsol)^4\big]\cell \\
  &\qquad + \frac1{4\sigma^4}\max_k|\dff_k\phi|^4
    \sum_{\kappa=\hi}^{k^*_{i+1}-\hi}\dsol_\kappa^{5-3\alpha}\big[
    (\dsol_{\kappa+1}^\alpha+\dsol_\kappa^\alpha)^4+(\dsol_{\kappa-1}^\alpha+\dsol_\kappa^\alpha)^4
    \big]\cell \\
  &\le 2\sigma\dss_{i+1}(\dsol)+ \frac{12\|\Phi\|_{C^1}^4}{\sigma^4}\rho_{i+1}^{-4}\aux_{i+1}(\dsol),
\end{align*}
where we have used \eqref{eq:FbyDiss}, that $\dsol_{k_{i+1}^*-\hi}=0$,
and that
\begin{align*}
  \dsol_\kappa^{5-3\alpha}\big[
  (\dsol_{\kappa+1}^\alpha+\dsol_\kappa^\alpha)^4+(\dsol_{\kappa-1}^\alpha+\dsol_\kappa^\alpha)^4
  \big]
  \le 8\dsol_\kappa^{5-3\alpha}\big[\dsol_{\kappa+1}^{4\alpha}+2\dsol_\kappa^{4\alpha}+\dsol_{\kappa-1}^{4\alpha}\big] 
  \le 8\big[\dsol_{\kappa+1}^{5+\alpha}+4\dsol_\kappa^{5+\alpha}+\dsol_{\kappa-1}^{5+\alpha}\big].
\end{align*}
Summarizing our results so far, we have shown that
\begin{equation}
  \label{eq:intermediate}
  \begin{split}
    J&\ge  (1+\alpha) \sum_\kappa\phi_\kappa\Big(
    \dsol_\kappa^{3+\alpha}(\Delta_\kappa\dsol)^2 
    -\frac\alpha4 \dsol_\kappa^{1+\alpha}\big[(\dff_\kapph\dsol)^2+(\dff_\kapmh\dsol)^2\big]^2 \\
    &\qquad\qquad\qquad\qquad\qquad -\frac{1+4\alpha}4\dsol_\kappa^{2+\alpha}|\Delta_\kappa\dsol|
    \big[(\dff_\kapph\dsol)^2+(\dff_\kapmh\dsol)^2\big]
    \Big)\cell
    \\&\qquad
    - C'\sigma\dss_{i+1}(\dsol) - B'\sigma^{-4}\rho_{i+1}^{-4}\aux_{i+1}(\dsol),
  \end{split}
\end{equation}
with positive constants $B'$ and $C'$ that are expressible in terms of the norms of $\Phi$ alone.

\subsubsection{Proof of \eqref{eq:tfdiss} --- summation by parts}
In this section, we derive the essential summation by parts rule for further estimation of the dissipation.
It is a spatially discrete variant of the following identity for smooth functions $Z,\varphi:[0,1]\to\setR_{\ge0}$,
subject to homogeneous Dirichlet boundary conditions:
\begin{align*}
  0 &= \int_0^1 \partial_\xi\big[\varphi Z^{2+\alpha} (\partial_\xi Z)^3\big]\dd\xi\\
  &= \int_0^1 \partial_\xi\varphi Z^{2+\alpha}(\partial_\xi Z)^3\dd\xi
  + (2+\alpha) \int_0^1 \varphi Z^{1+\alpha}(\partial_\xi Z)^4\dd\xi
  + 3\int_0^1\varphi Z^{2+\alpha}\partial_{\xi\xi}Z(\partial_\xi Z)^2\dd\xi.
\end{align*}
This formula plays the key role in the derivation of \eqref{eq:ZestTF}.
Our translation to the grid functions $\dsol$ and $\phi$ is this:
\begin{align*}
  0 = \sum_\kappa \left[
      \frac{\phi_{\kappa+1}\dsol_{\kappa+1}^{2+\alpha}+\phi_\kappa\dsol_\kappa^{2+\alpha}}2(\dff_{\kapph}\dsol)^3
      -
      \frac{\phi_{\kappa-1}\dsol_{\kappa-1}^{2+\alpha}+\phi_\kappa\dsol_\kappa^{2+\alpha}}2(\dff_{\kapmh}\dsol)^3
  \right]
  = \hat S_1+\hat S_2,
\end{align*}
where
\begin{align*}
  \hat S_1&:=\sum_\kappa \phi_\kappa\dsol_\kappa^{2+\alpha}\big[(\dff_{\kapph}\dsol)^3-(\dff_{\kapmh}\dsol)^3\big], \\
  \hat S_2&:=\frac12\sum_\kappa \left[
      \frac{\phi_{\kappa+1}\dsol_{\kappa+1}^{2+\alpha}-\phi_\kappa\dsol_\kappa^{2+\alpha}}2(\dff_{\kapph}\dsol)^3
      +
      \frac{\phi_{\kappa}\dsol_{\kappa}^{2+\alpha}-\phi_{\kappa-1}\dsol_{\kappa-1}^{2+\alpha}}2(\dff_{\kapmh}\dsol)^3
  \right].
\end{align*}
The first sum is simple to estimate from below:
\begin{align*}
  \hat S_1 &= \sum_\kappa \phi_\kappa\dsol_\kappa^{2+\alpha}
        \Delta_\kappa\dsol\big[(\dff_\kapph\dsol)^2+(\dff_\kapph\dsol)(\dff_\kapmh\dsol)+(\dff_\kapmh\dsol)^2\big]\cell \\
      &\ge -\frac32\sum_\kappa \phi_\kappa\dsol_\kappa^{2+\alpha}
        \big|\Delta_\kappa\dsol\big|\big[(\dff_\kapph\dsol)^2+(\dff_\kapmh\dsol)^2\big]\cell.
\end{align*}
The second sum gives the significant contribution,
which is extracted by means of \eqref{eq:evenmoreelementary2}:
\begin{align*}
  \hat S_2 &\ge \frac12\sum_\kappa \left[
        \big(\phi_{\kappa+1}\dsol_{\kappa+1}^{1+\alpha}+\phi_\kappa\dsol_\kappa^{1+\alpha}\big)(\dff_{\kapph}\dsol)^4
        +
        \big(\phi_{\kappa}\dsol_{\kappa}^{1+\alpha}+\phi_{\kappa-1}\dsol_{\kappa-1}^{1+\alpha}\big)(\dff_{\kapmh}\dsol)^4
        \right]\cell
        - \hat S_3 \\
      & = \sum_\kappa \phi_\kappa\dsol_\kappa^{1+\alpha}\big[(\dff_{\kapph}\dsol)^4+(\dff_{\kapmh}\dsol)^4\big]\cell
        - \hat S_3,
\end{align*}
where $\hat S_3$ is used to collect the reminder terms from \eqref{eq:evenmoreelementary2}.
More specifically, one has, using Young's inequality with exponents $4$ and $\frac43$,
\begin{align*}
  \hat S_3 &:=
        \frac12\sum_\kappa |\dff_{\kapph}\phi|\big(\dsol_{\kappa+1}^{2+\alpha}+\dsol_\kappa^{2+\alpha}\big)|\dff_{\kapph}\dsol|^3\cell +
        \frac12\sum_\kappa |\dff_{\kapmh}\phi|\big(\dsol_{\kappa-1}^{2+\alpha}+\dsol_\kappa^{2+\alpha}\big)|\dff_{\kapmh}\dsol|^3\cell \\
      &= \sum_\kappa \dsol_\kappa^{2+\alpha}
        \big[|\dff_{\kapph}\phi||\dff_{\kapph}\dsol|^3+|\dff_{\kapmh}\phi||\dff_{\kapmh}\dsol|^3\big]\cell \\
      &\le \frac{3\sigma}4\sum_{\kappa=\hi}^{k^*_{i+1}-\hi}
        \dsol_\kappa^{1+\alpha}\big[(\dff_{\kapph}\dsol)^4+(\dff_{\kapmh}\dsol)^4\big]\cell
        + \frac1{2\sigma^{3}}\max_k|\dff_k\phi|^4\sum_{\kappa=\hi}^{k^*_{i+1}-\hi}\dsol_\kappa^{5+\alpha}\cell \\
      &\le \frac{3\sigma}2\dss_{i+1}(\dsol) + \frac{\|\Phi\|_{C^1}^4}{2\sigma^{3}}\rho_{i+1}^{-4}\aux_{i+1}.
\end{align*}
Summarizing, we obtain that
\begin{equation}
  \label{eq:sumbyparts}
  \begin{split}
    0&\ge\sum_\kappa \phi_\kappa\dsol_\kappa^{1+\alpha}\big[(\dff_{\kapph}\dsol)^4+(\dff_{\kapmh}\dsol)^4\big]
    -\frac32\sum_\kappa \phi_\kappa\dsol_\kappa^{2+\alpha}
    \big|\Delta_\kappa\dsol\big|\big[(\dff_\kapph\dsol)^2+(\dff_\kapmh\dsol)^2\big]\cell \\
    & \qquad -C''\sigma\dss_{i+1}(\dsol) - B''\sigma^{-4}\rho_{i+1}^{-4}\aux_{i+1}(\dsol)
  \end{split}
\end{equation}
with positive constants $B''$ and $C''$ that are again expressible in terms of the norms of $\Phi$ alone.

\subsubsection{Proof of \eqref{eq:tfdiss} --- conclusion}
We return to \eqref{eq:intermediate},
and add $(1+\alpha)/3$ times the expression on the right-hand side of \eqref{eq:sumbyparts}.
Since
\begin{align*}
  &\dsol_\kappa^{3+\alpha}(\Delta_\kappa\dsol)^2 
    +\left(\frac13-\frac\alpha2\right) \dsol_\kappa^{1+\alpha}\big[(\dff_\kapph\dsol)^4+(\dff_\kapmh\dsol)^4\big] \\
  &\qquad  -\frac{3+4\alpha}4\dsol_\kappa^{2+\alpha}|\Delta_\kappa\dsol|
    \big[(\dff_\kapph\dsol)^2+(\dff_\kapmh\dsol)^2\big] \\
  &\ge
    \dsol_\kappa^{1+\alpha}\left(\frac{3+4\alpha}8\right)\left(
    \sqrt{\frac52}\dsol_\kappa|\Delta_\kappa\dsol|
    -\sqrt{\frac25}\big[(\dff_\kapph\dsol)^2+(\dff_\kapmh\dsol)^2\big]
    \right)^2 \\
  &\qquad +\left(\frac1{16}-\frac54\alpha\right)\dsol_\kappa^{3+\alpha}(\Delta_\kappa\dsol)^2
    +\left(\frac1{30}-\frac9{10}\alpha\right)\dsol_\kappa^{1+\alpha}\big[(\dff_\kapph\dsol)^4+(\dff_\kapmh\dsol)^4\big] \\
  &\ge
    \left(\frac1{16}-2\alpha\right) \left(
    \dsol_\kappa^{3+\alpha}(\Delta_\kappa\dsol)^2
    +\frac12 \big[(\dff_\kapph\dsol)^4+(\dff_\kapmh\dsol)^4\big]
    \right),
\end{align*}
we obtain eventually
\begin{align*}
  J \ge c\dss_i(\dsol)
  - C\sigma\dss_{i+1}(\dsol)
  - B\sigma^{-4}\rho_{i+1}^{-4}\aux_{i+1}(\dsol),
\end{align*}
with
\begin{align*}
  c := \frac1{20}-2\alpha, \quad
  B:= B'+\frac{1+\alpha}3B'', \quad
  C:=C'+\frac{1+\alpha}3C''.
\end{align*}
To conclude \eqref{eq:tfdiss} from here, it suffices to integrate the estimate above in time
from $t=0$ to $t=T$,
using that
\begin{align*}
  \int_0^T J\dd t
  = \sum_\kappa\frac{\bar\dsol_\kappa^\alpha}{\alpha}\phi_\kappa\cell
  - \sum_\kappa\frac{\dsol_\kappa(t)^\alpha}{\alpha}\phi_\kappa\cell
  \le \ent_{i+1}(\bar z) - \ent_i\big(\dsol(t)\big).
\end{align*}
The respective estimate \eqref{eq:tfdissI} for $i=I$ is obtained in an analogous manner,
but is easier since one has $\phi\equiv1$,
so that there are no contributions related to $\phi$ and its derivatives.

\subsection{The Stampacchia iteration}
Introduce
\begin{align*}
  \theta:=\frac1{1+\frac45\alpha}.
\end{align*}
In analogy to \eqref{eq:defa} and \eqref{eq:defb},
we consider
\begin{align*}
  a_i &:= \rho_i^{-1/\theta}\left[
        B\sigma^{-4}\rho_i^{-4}\int_0^T\aux_i\big(\dsol(t)\big) + C\sigma\int_0^T\dss_i\big(\dsol(t)\big)\dd t
        \right],\\
  b &:= \max_{i=2,\ldots,I}\left(\rho_{i}^{-1/\theta}\ent_i(\bar\dsol)\right).
\end{align*}
We are going to derive an iteration that is similar to (but more complicated than) the one in \eqref{eq:stampacchiaPME}.
In terms of $a_i$ and $b$, the dissipation relation \eqref{eq:tfdiss} yields
\begin{align}
  \label{eq:tfdiss2}
  \sup_{t\in[0,T]}\ent_i\big(\dsol(t)\big) + c\int_0^T\dss_i\big(\dsol(t)\big)\dd t
  \le \rho_{i+1}^{1/\theta}(a_{i+1}+b).
\end{align}
Thanks to the GNS inequality \eqref{eq:dGNS4}
with $s=\alpha/(5+\alpha)$, $k^*=k^*_i$ and $f_\kappa=\dsol_\kappa^{\frac{5+\alpha}4}$,
\begin{align*}
  \aux_i(\dsol)
  = \sum_{\kappa=\hi}^{k_i^*-\hi}\big(\dsol_\kappa^{\frac{5+\alpha}4}\big)^4\cell 
  \le A_{4,s}\left(\sum_{k=1}^{k_*-1}\big|\dff_k\big(\dsol^{\frac{5+\alpha}4}\big)\big|^4\cell\right)^\theta
  \left(\sum_{\kappa=\hi}^{k_i^*-\hi}\dsol_\kappa^\alpha\cell\right)^{1+3\theta}.
\end{align*}
By the elementary estimate \eqref{eq:diffquot} from the appendix,
and recalling that $(x+y)^p\le x^p+y^p$ any $p\in(0,1)$, and for all positive real numbers $x,y$,
\begin{align*}
  \big[\dff_k\big(\dsol^{\frac{5+\alpha}4}\big)\big]^4
  \le C_\alpha\big(\dsol_{k+\hi}^{1+\alpha}+\dsol_{k-\hi}^{1+\alpha}\big)[\dff_k\dsol]^4,
\end{align*}
where $C_\alpha>0$ depends only on $\alpha$.
And so,
\begin{align*}
  \aux_i(\dsol) \le A_{4,s} \alpha^{1+3\theta} (2C_\alpha)^\theta\dss_i(\dsol)^\theta\ent_i(\dsol)^{1+3\theta}.
\end{align*}
Integration in time, an application of H\"older's inequality,
and substitution of \eqref{eq:tfdiss2} yield
\begin{align*}
  \int_0^T\aux_i\big(\dsol(t)\big)\dd t
  &\le A_{4,s} \alpha^{1+3\theta}\left(\frac{2C_\alpha}c\right)^\theta \sup_{t\in[0,T]}\ent_i\big(\dsol(t)\big)^{1+3\theta}
    \left(c\int_0^T\dss_i\big(\dsol(t)\big)\dd t\right)^\theta T^{1-\theta} \\
  &\le  A_{4,s} \alpha^{1+3\theta}\left(\frac{2C_\alpha}c\right)^\theta T^{1-\theta}\rho_{i+1}^{4+1/\theta}(a_{i+1}+b)^{1+4\theta}.
\end{align*}
On the other hand, it is a trivial consequence of \eqref{eq:tfdiss2} that
\begin{align*}
  \int_0^T\dss_i\big(\dsol(t)\big)\dd t \le \frac1c\rho_{i+1}^{1/\theta}(a_{i+1}+b).
\end{align*}
A combination of these two estimates
--- recalling that $\rho_i=2\rho_{i+1}$ for the equi-distant mesh, and using that $1-\theta=\frac45\alpha\theta$ --- as well as \eqref{eq:tfdissI} for $i=I$
yields the recursion relation
\begin{align*}
  &a_I \le c_1 b + c_2b^{1+4\theta}, \qquad
  a_i \le c_1 (a_{i+1}+b)+ c_2(a_{i+1}+b)^{1+4\theta} \quad\text{for $i=2,3,\ldots,I-1$}, \\
  &\text{with}\quad
  c_1:=\frac{2^{1/\theta}C}c\sigma,\quad
  c_2:=A_{4,s} 2^{4+1/\theta}\alpha^{1+3\theta}\left(\frac{2C_\alpha}c\right)^\theta B\sigma^{-4}T^{\frac45\alpha\theta}.
\end{align*}
One readily checks that the choices
\begin{align}
\label{ChoiceOfT}
  \sigma:=\frac14\left(\frac{2^{1/\theta}C}c\right)^{-1},
  \quad
  T:=\left(A_{4,s} 2^{6+4\theta+1/\theta}\alpha^{1+3\theta}\left(\frac{2C_\alpha}c\right)^\theta B\sigma^{-4}\right)^{-\frac5{4\alpha\theta}}b^{-\frac5{\alpha}}
\end{align}
imply that
\begin{align*}
  c_1=\frac14, \quad c_2=\frac14(2b)^{-4\theta}.
\end{align*}
An induction argument now shows that $a_i\le b$ for all $i=2,3,\ldots,I$.
Indeed,
\begin{align*}
  a_I \le \frac14b+\frac14 2^{-4\theta}b \le b,
\end{align*}
and if $a_{i+1}\le b$, then also
\begin{align*}
  a_i \le \frac14(2b)+\frac14(2b)^{-4\theta}(2b)^{1+4\theta} = b.
\end{align*}
So, in particular, for the choice of $T$ as in \eqref{ChoiceOfT} we get
\begin{align}
  \label{eq:crucialTF}
  B\sigma^{-4}\edge_2^{-(5+\frac45\alpha)} \int_0^T\big(z_{\nicefrac{1}{2}}^{5+\alpha}+z_{\nicefrac{3}{2}}^{5+\alpha}\big) \dd t\,\cell
  \le b.
\end{align}

\subsection{End of the proof of Theorem \ref{thm:dtfmain}}
From \eqref{eq:devolTF} and the boundary conditions \eqref{eq:dhomDiri}\&\eqref{eq:dhomNeum}
we obtain the following evolution equation for the position $x_0(t)$ of the left edge of support: 
\begin{align*}
  \dot x_0 = \frac{\dsol_{\hi}^2}{4\cell^3}\big(z_{\nicefrac32}^2+2\dsol_\hi\dsol_{\nicefrac32}-6\dsol_\hi^2\big),
\end{align*}
and consequently,
\begin{align*}
  |\dot x_0| \le 8\cell^{-3}\big(z_{\nicefrac32}^4+\dsol_\hi^4\big).
\end{align*}
For any $t^*\in[0,T]$, it follows thanks to \eqref{eq:crucialTF} that
\begin{align*}
  |x_0(t^*)-a|^{\frac{5+\alpha}4}
  &\le \left(\int_0^{t^*}|\dot x_0(t)|\dd t\right) ^{\frac{5+\alpha}4} \\
  &\le \big(8\cell^{-3}\big) ^{\frac{5+\alpha}4}\left(\int_0^{t^*}\big(z_{\nicefrac32}^4+\dsol_\hi^4\big)\dd t\right) ^{\frac{5+\alpha}4} \\
  & \le C_\alpha\cell^{-\frac34(5+\alpha)}(t^*)^{\frac{1+\alpha}4}\int_0^T \big(z_{\nicefrac32}^{5+\alpha}+\dsol_\hi^{5+\alpha}\big)\dd t \\
  & \le C_\alpha \cell^{-\frac34(5+\alpha)+(5+\frac45\alpha)-1}\sigma^4B^{-1}b(t^*)^{\frac{1+\alpha}4}
  = C_\alpha\sigma^4B^{-1}\cell^{\frac14+\frac\alpha{20}}b(t^*)^{\frac{1+\alpha}4},
\end{align*}
and hence
\begin{align*}
  |x_0(t^*)-a|\le C\cell^{1/5}\big[b^4(t^*)^{1+\alpha}\big]^{\frac1{5+\alpha}}.
\end{align*}
By the same argument as in the end of the proof of Theorem \ref{thm:dpmemain}, it follows that $b\le\bar b$.
Hence, the claim is proven.


\section{Numerical experiments}
In this short section, we present results from simple numerical simulations
in which the waiting time phenomenon is clearly visible.
Specifically, we consider
the discretized porous medium equation \eqref{eq:devolPME} with $m=2$,
and a variant of the discretized thin film equation \eqref{eq:devolTF} with non-equidistant grid.
For these, we study the discrete solutions corresponding to the initial data
\begin{align}
  \label{eq:dinit}
  \bar u(x) =
  \begin{cases}
    C_q\cos^q(\pi x) & \text{for $|x|<\frac12$}, \\
    0 & \text{for $|x|\ge\frac12$},
  \end{cases}
\end{align}
with different values of $q>0$, where $C_q$ is chosen to adjust $\bar u$'s mass to unity.
Our theory predicts the occurence of waiting times for $q\ge2$ in the case of the porous medium equation,
and for $q\ge4$ in case of the thin film equation.

For a given number $K$ of nodes, the discretization in mass space is defined as follows:
for $k=0,1,\ldots,K$, we choose the initial position $\bar x_k$ of the $k$th point as $\bar x_k=-\frac12\cos(\pi k/K)$
--- so that $\bar x_0=-\frac12$ and $\bar x_K=+\frac12$ mark the left and the right edge of $\bar u$'s support, respectively ---
and let $\xi_k := \int_{-\frac12}^{\bar x_k}\bar u(x)\dd x$ in accordance with \eqref{eq:consistent}.
This guarantees an improved resolution of $\bar u$ near the edges of support,
with a spatial mesh width of order $O(K^{-2})$ instead of the mesh width $O(K^{-1})$ in the bulk of $\bar u$.

Simulations have been performed for a variety of different choices of $q$ and $K$.
Qualitative results for $K=50$ and selected values of $q$ below and above the critical value
are reported in Figures \ref{fig:qualPME} and \ref{fig:qualTFE} for porous medium and thin film, respectively.
In both cases,
the top row shows an overlay of snapshots of the density in physical space at different instances of time,
the bottom row shows the position of the Lagrangian points $x_k(t)$ as functions of time.

\begin{figure}
  \centering
  \includegraphics[width=0.24\textwidth]{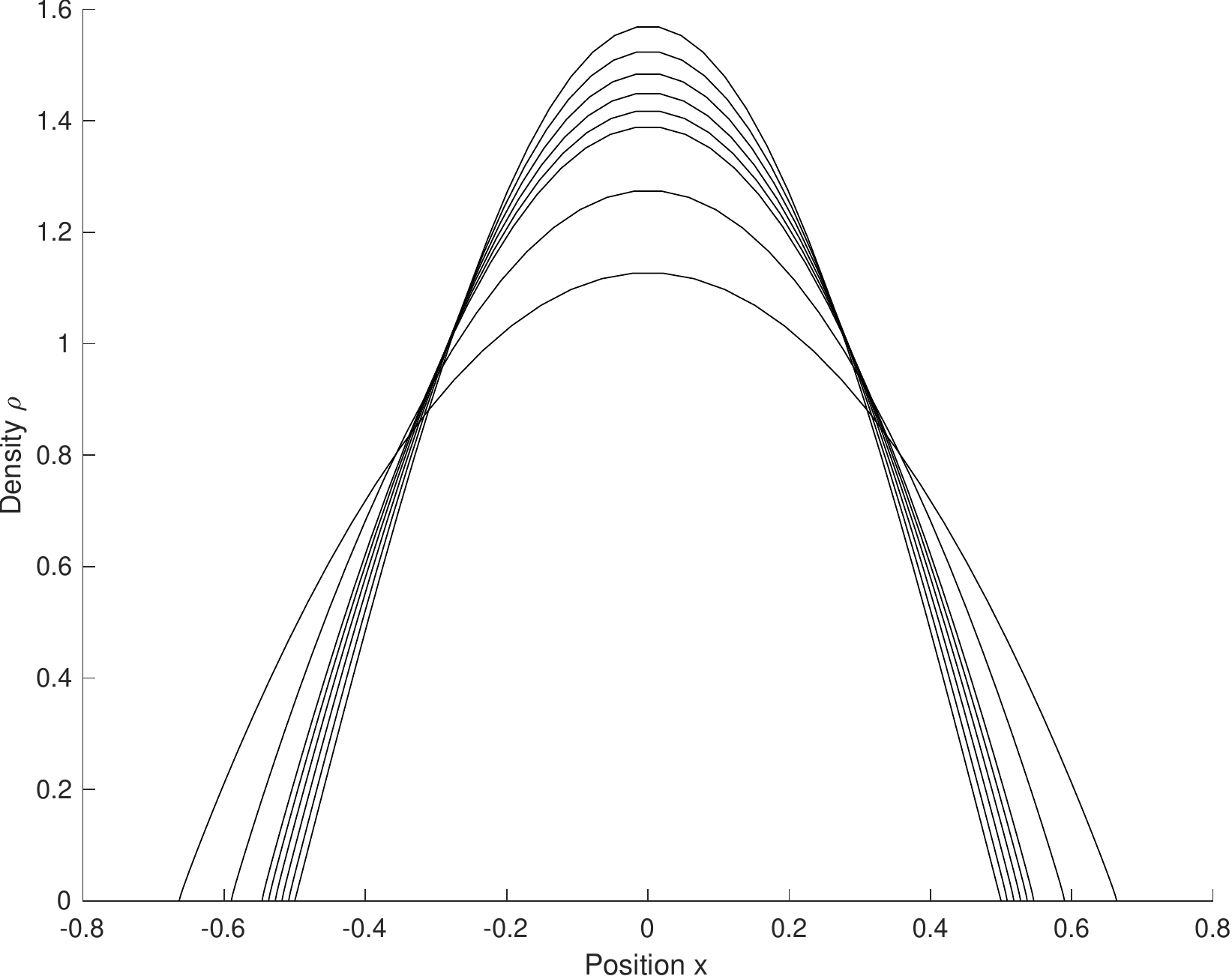}
  \includegraphics[width=0.24\textwidth]{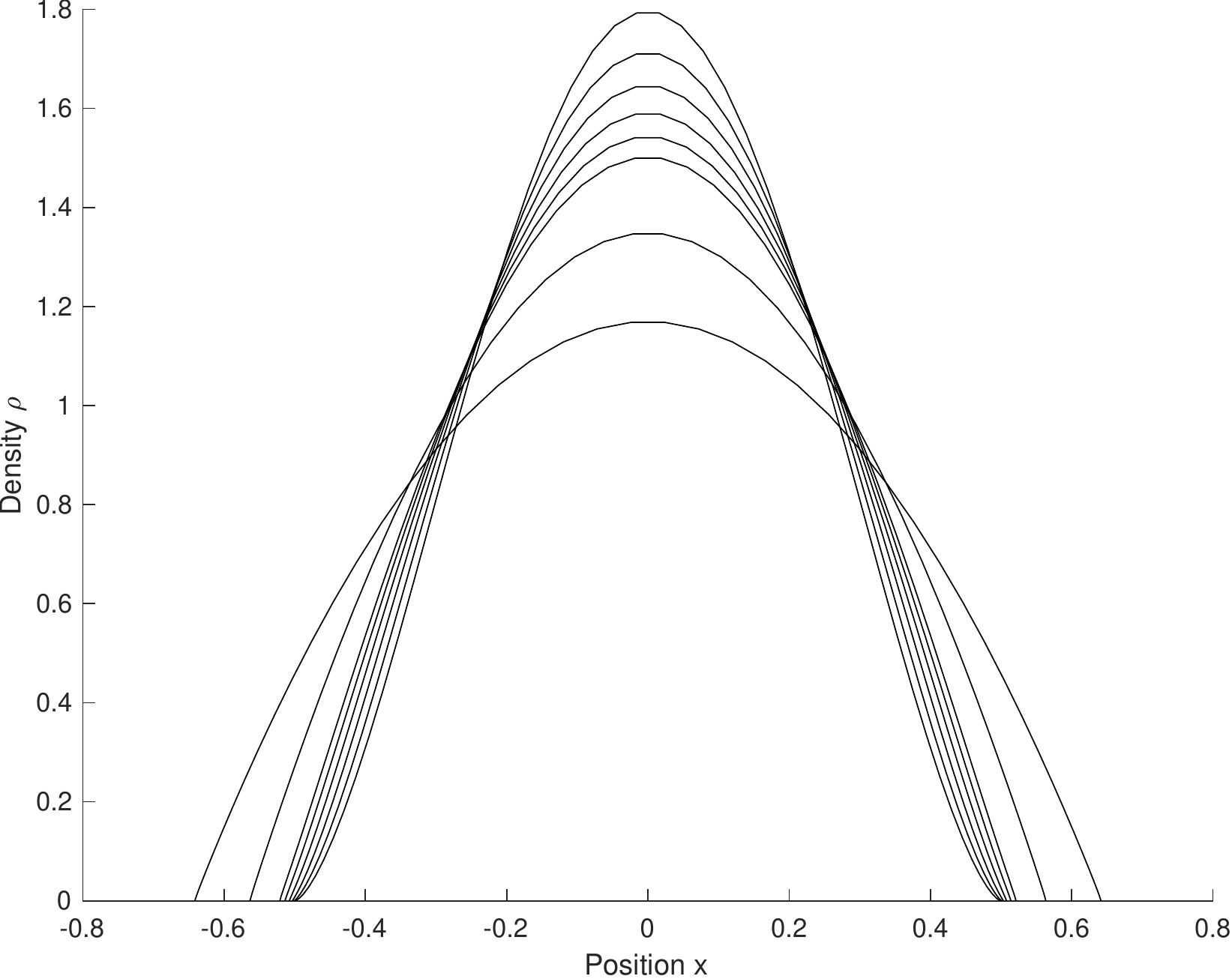}
  \includegraphics[width=0.24\textwidth]{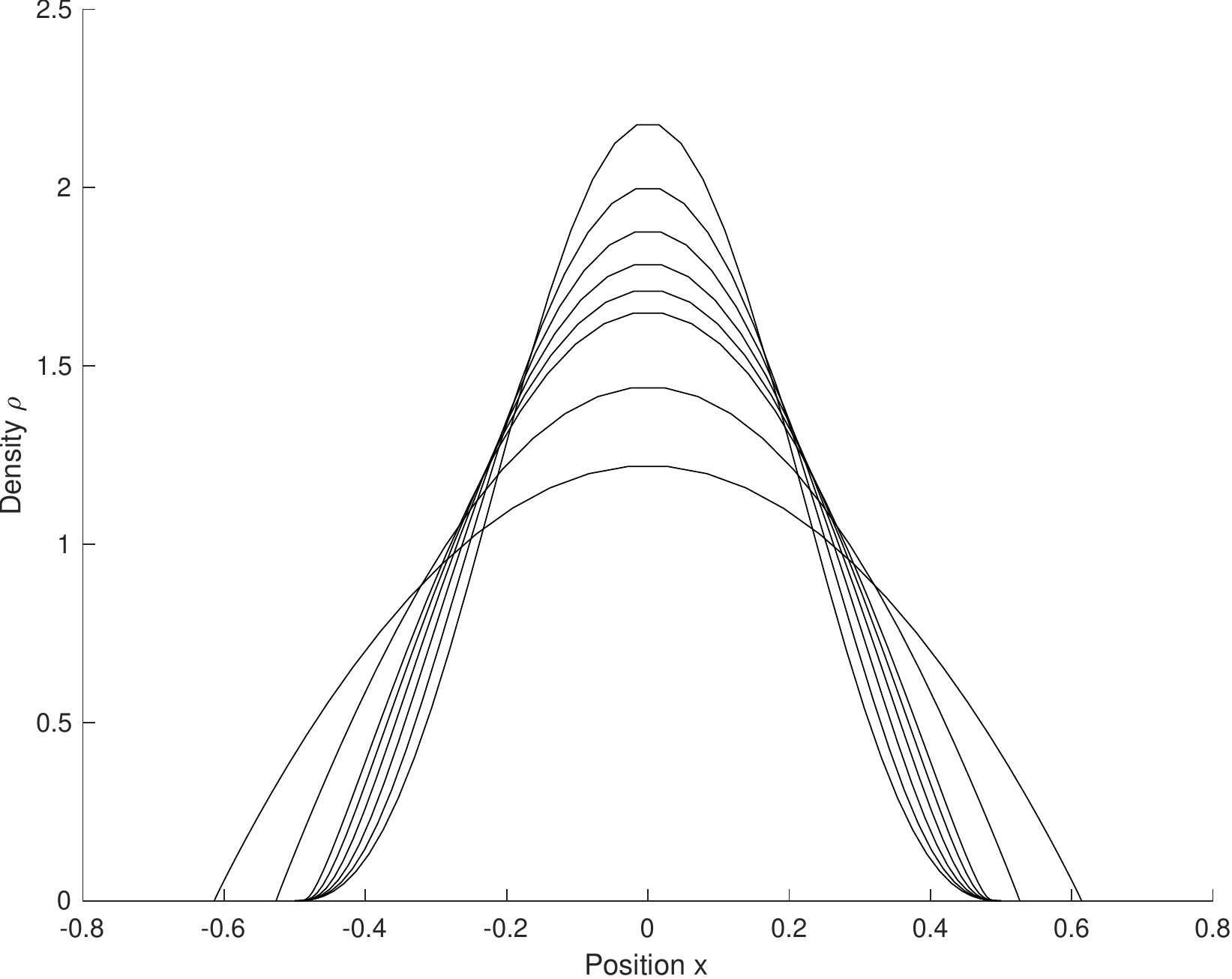}
  \includegraphics[width=0.24\textwidth]{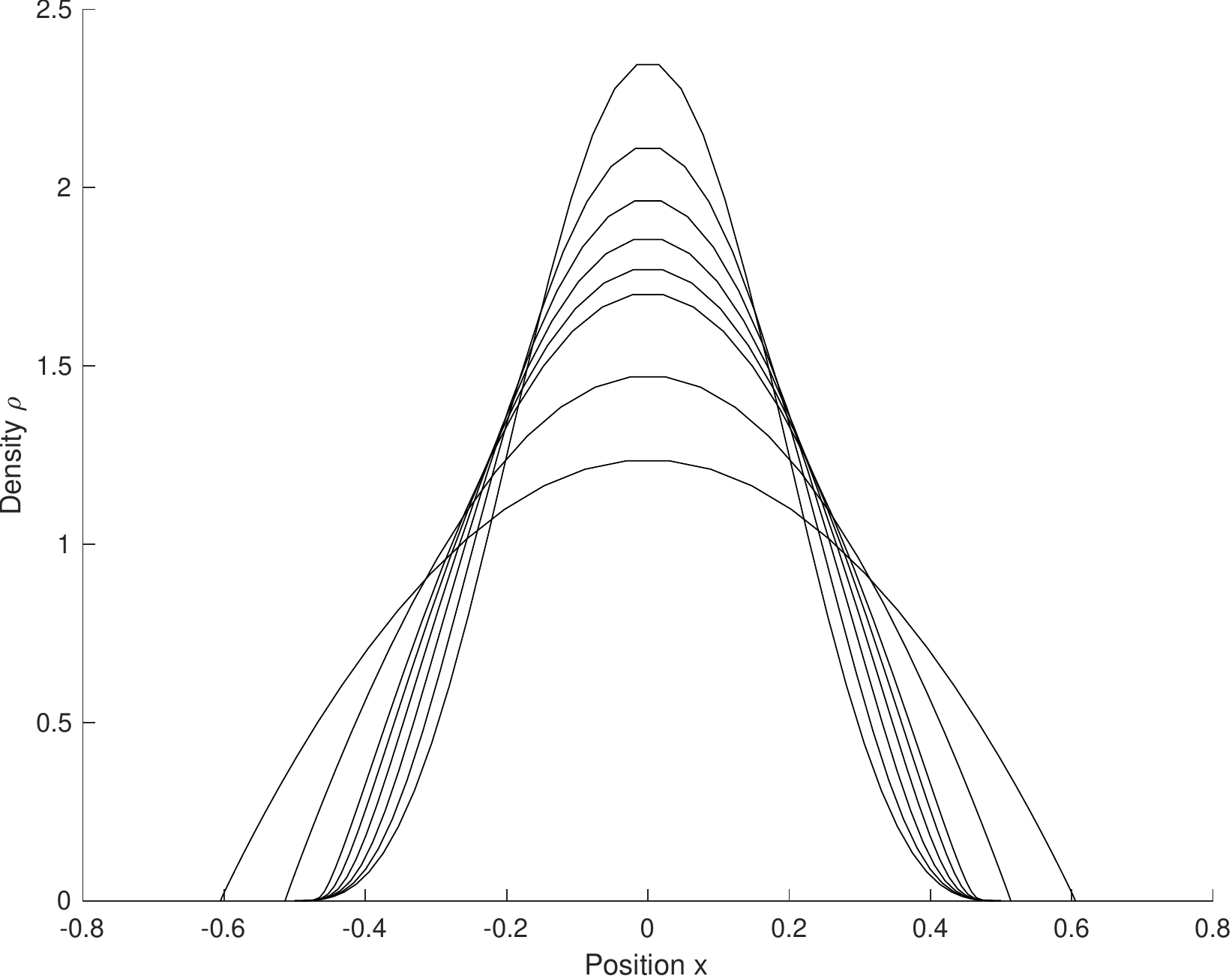}
  \newline
  \includegraphics[width=0.24\textwidth]{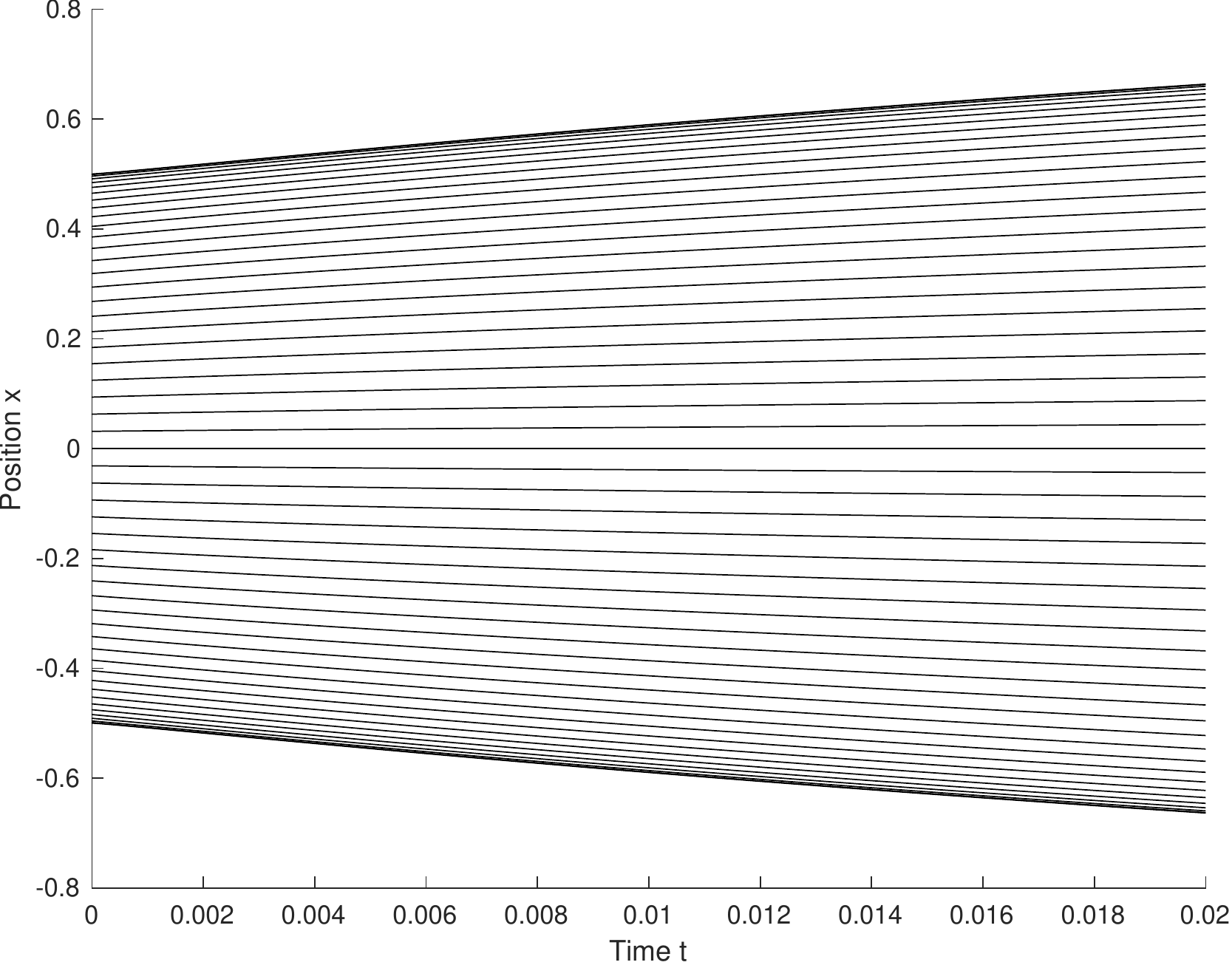}
  \includegraphics[width=0.24\textwidth]{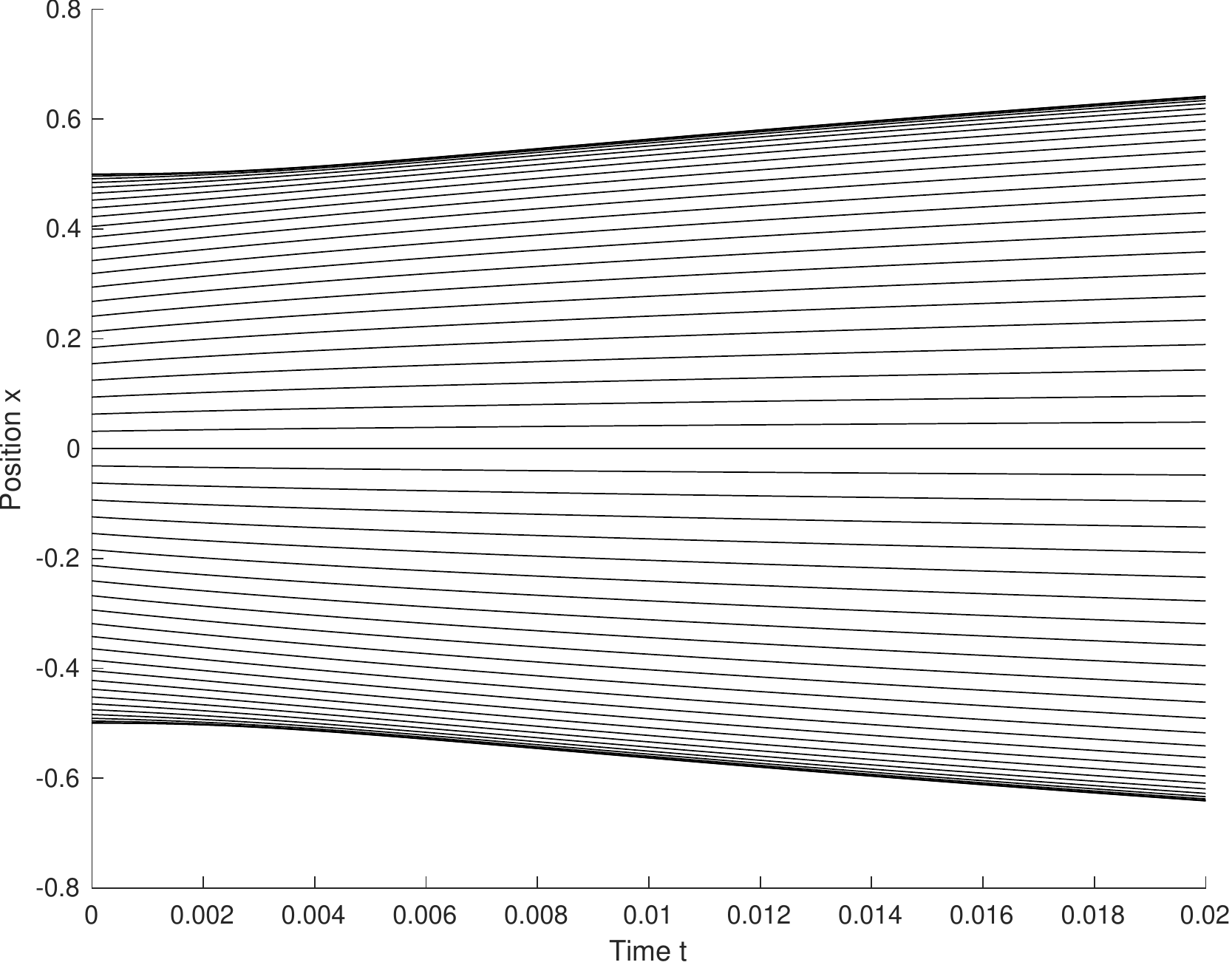}
  \includegraphics[width=0.24\textwidth]{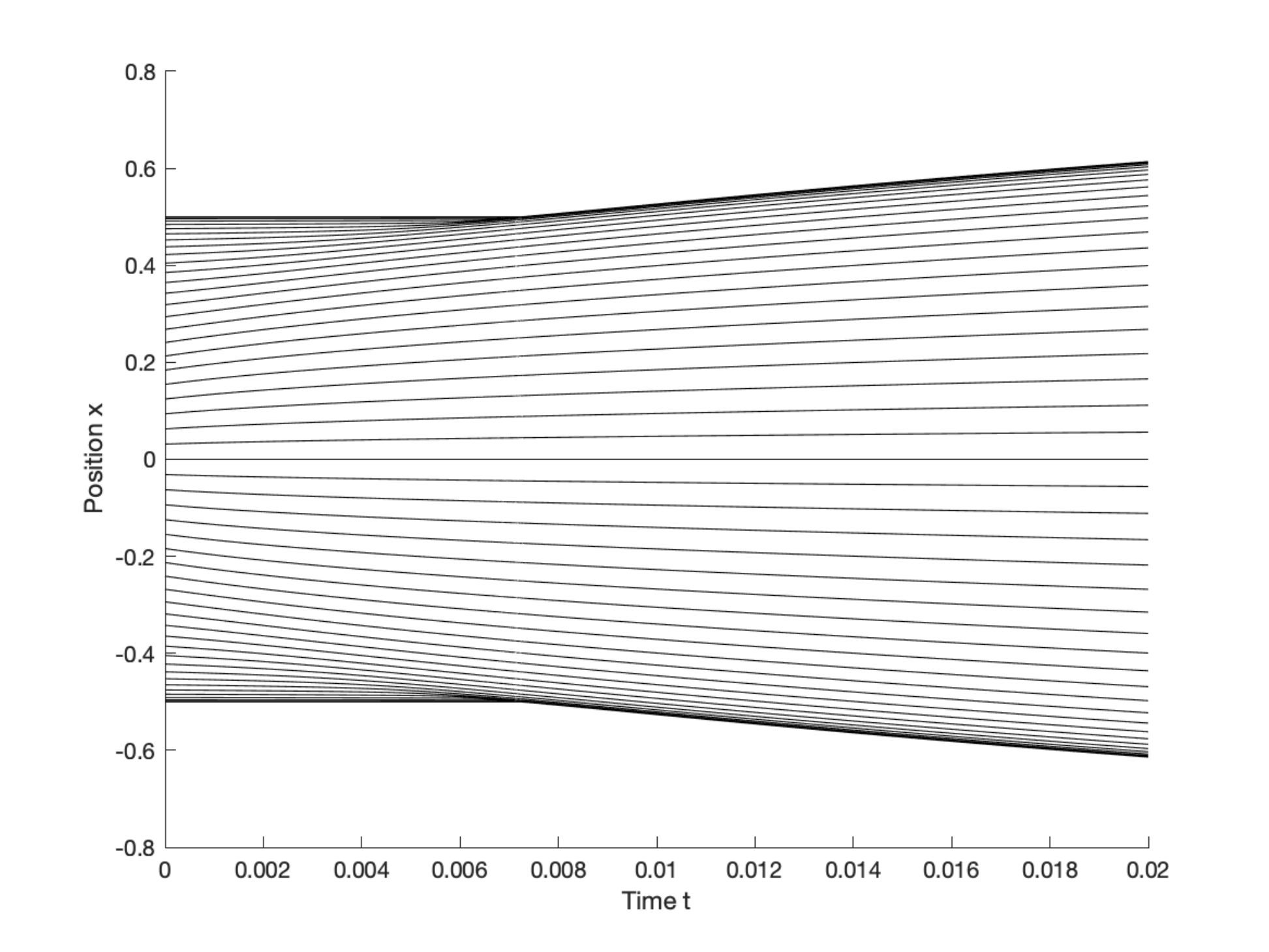}
  \includegraphics[width=0.24\textwidth]{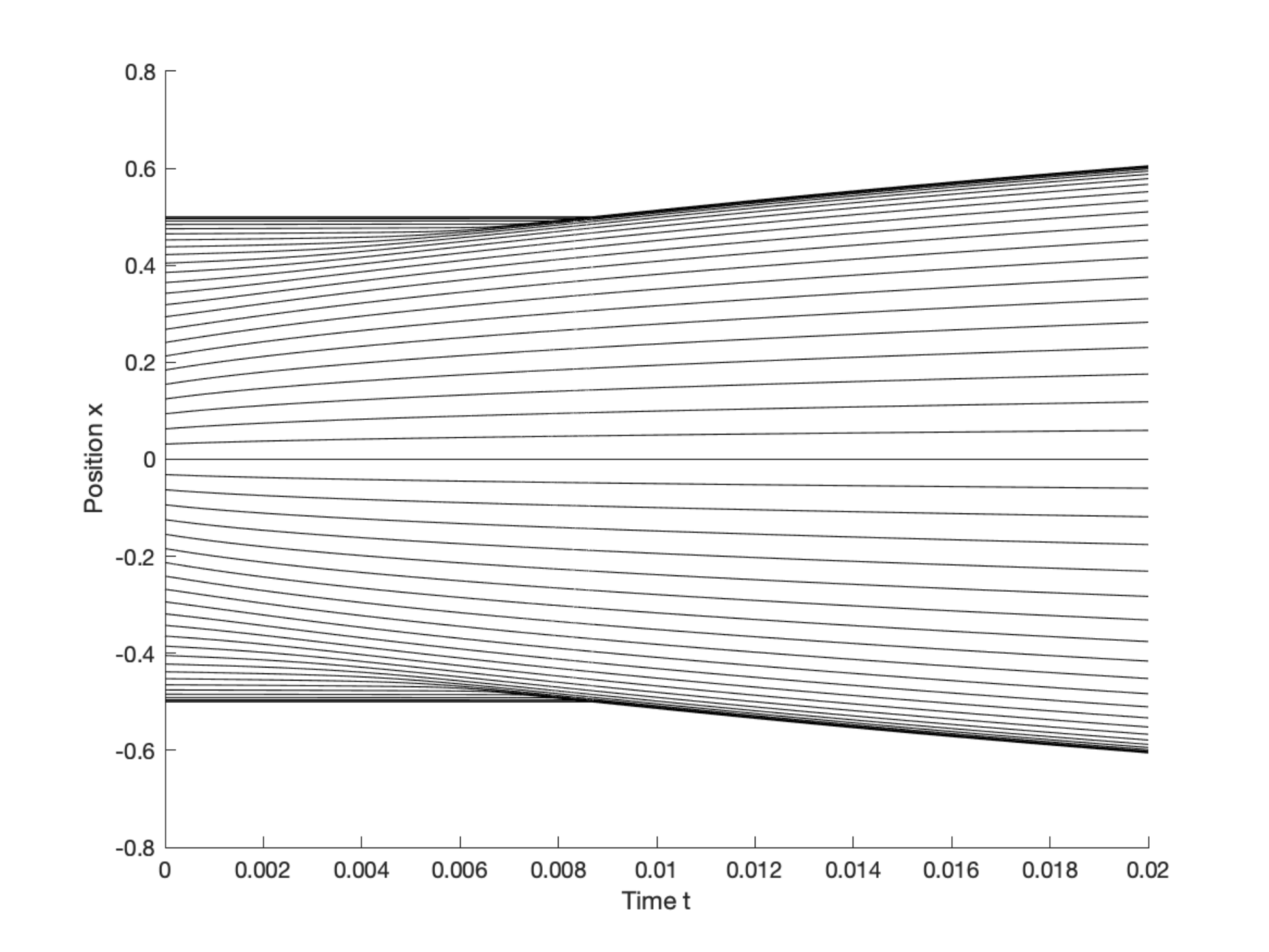}
  \caption{Qualitative illustration of the solutions to the discrete porous medium equation for $K=50$:
    for exponents $q=1.0,\,1.5,\,2.5,\,3.0$ from left to right,
    snapshots of the densities $\rho$
    at times $t=0,1,2,3,4,5,10,20\times 10^{-3}$ are shown on top,
    and trajectories of the Lagrangian points directly below.}
  \label{fig:qualPME}
\end{figure}

\begin{figure}
  \centering
  \includegraphics[width=0.24\textwidth]{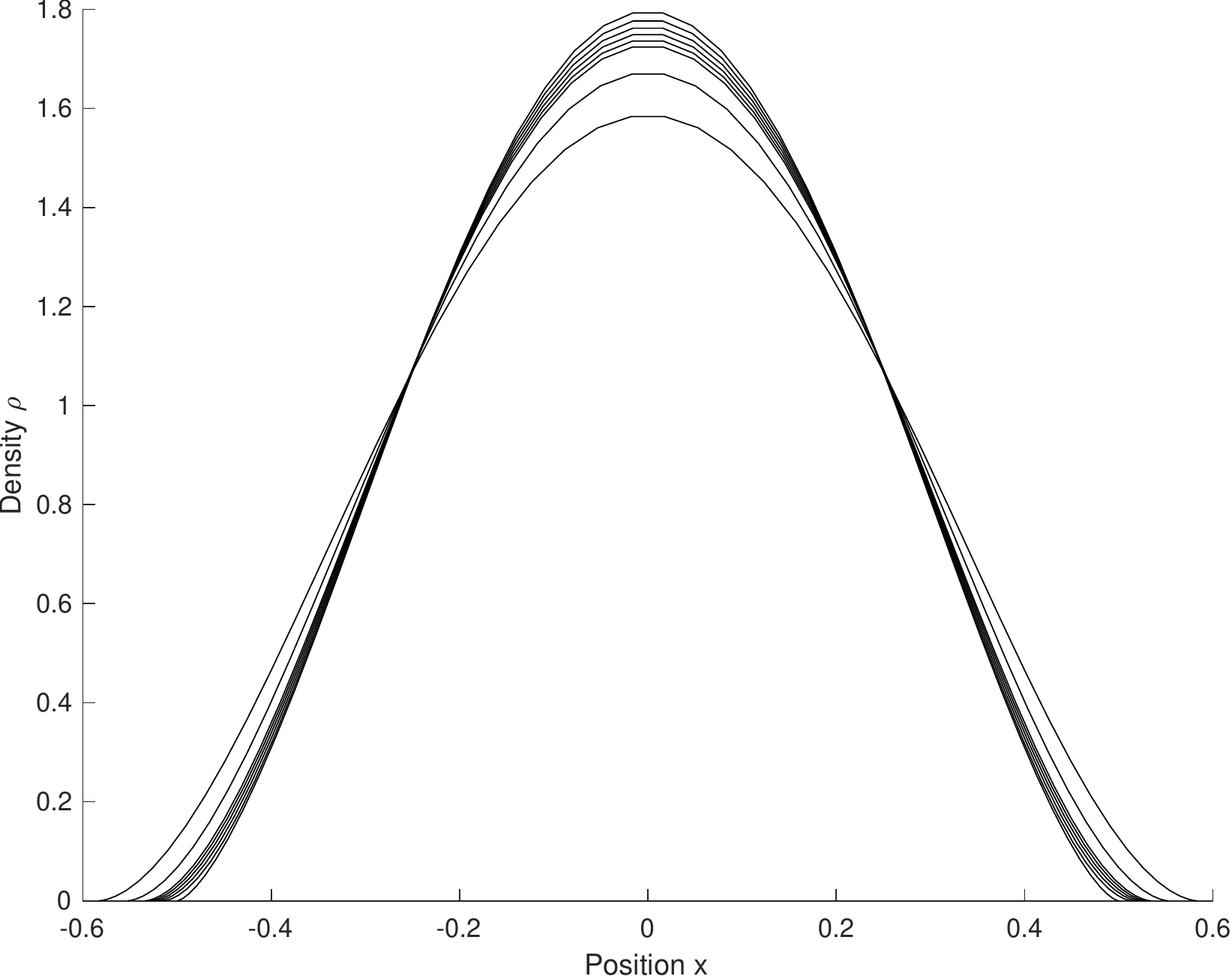}
  \includegraphics[width=0.24\textwidth]{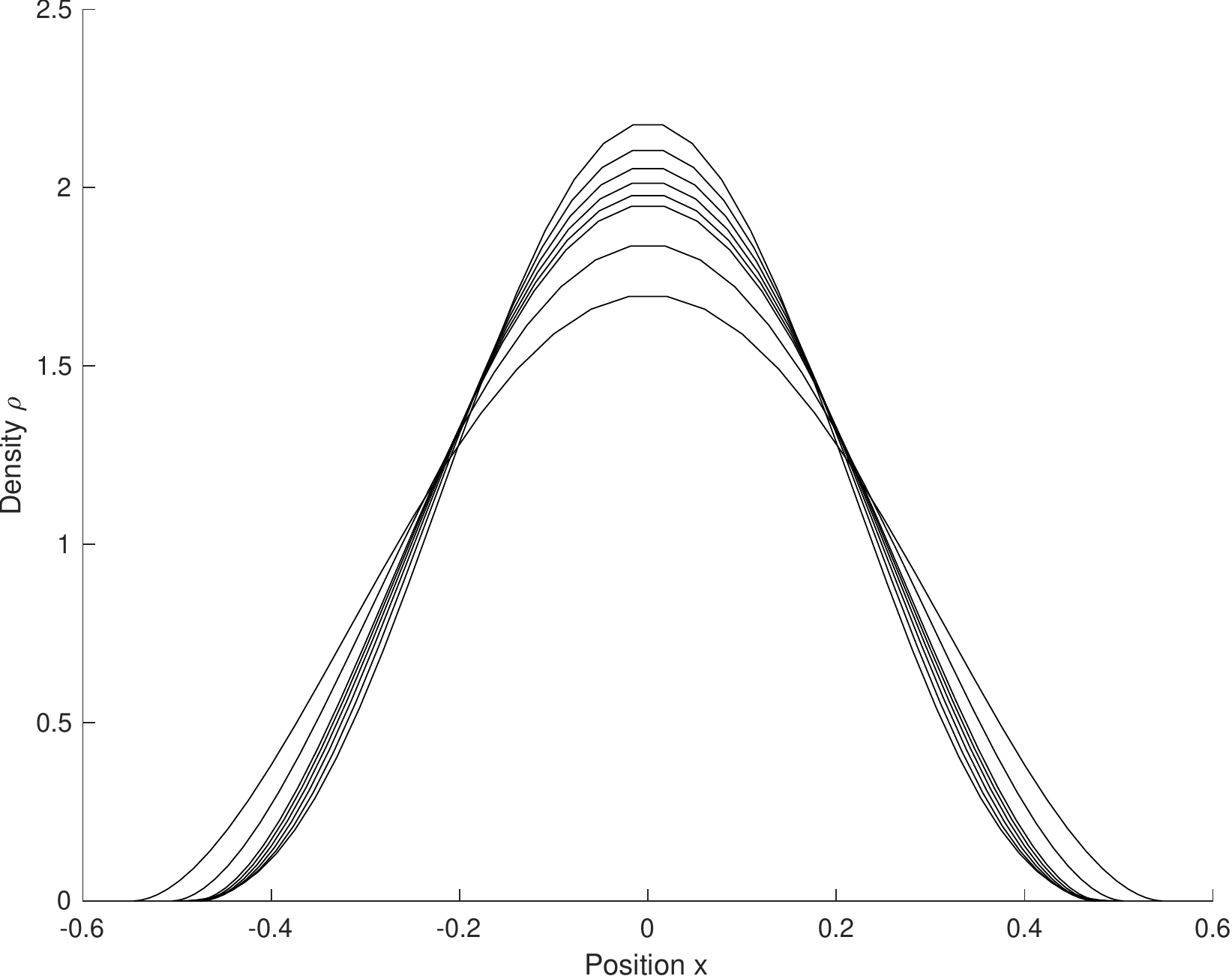}
  \includegraphics[width=0.24\textwidth]{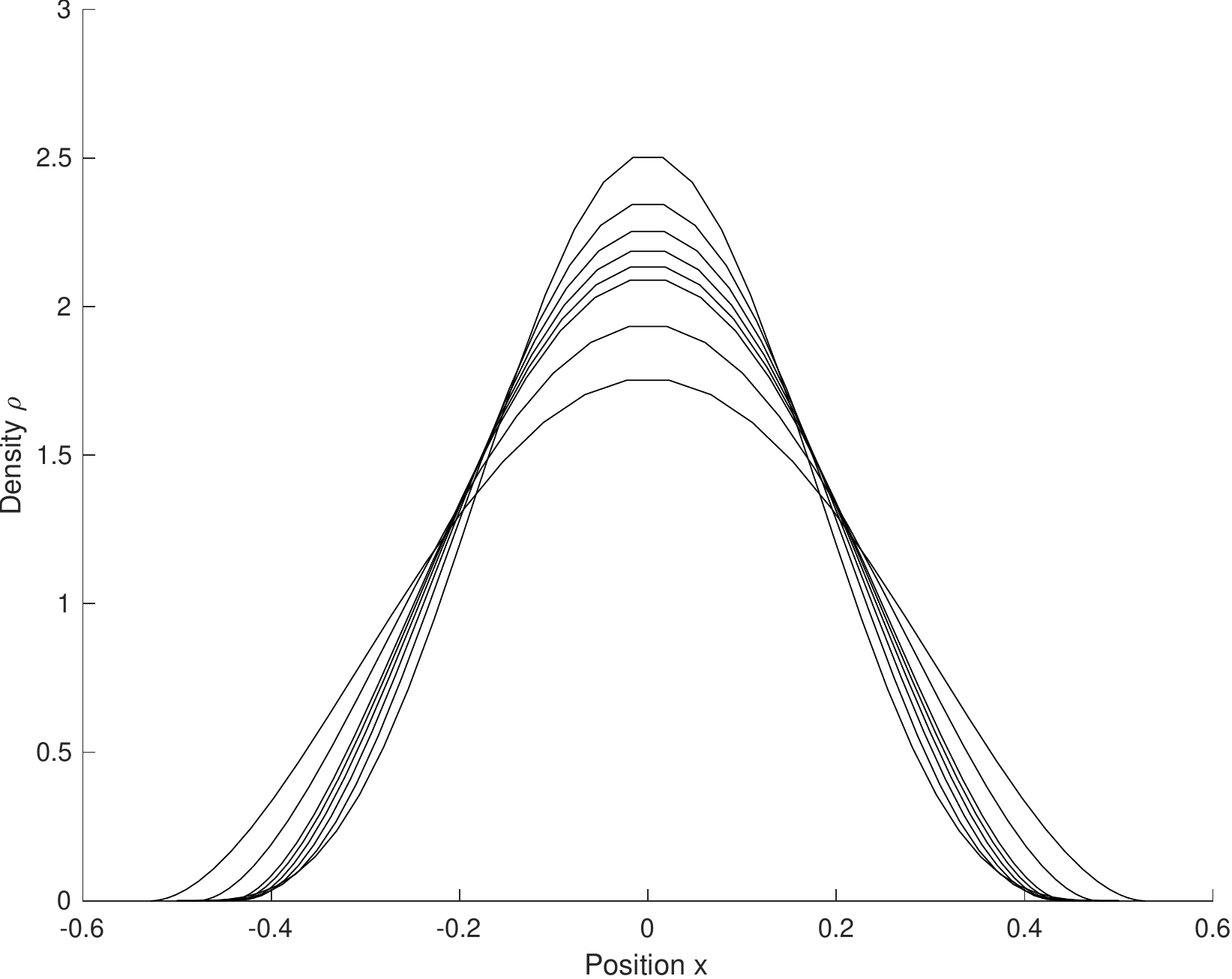}
  \includegraphics[width=0.24\textwidth]{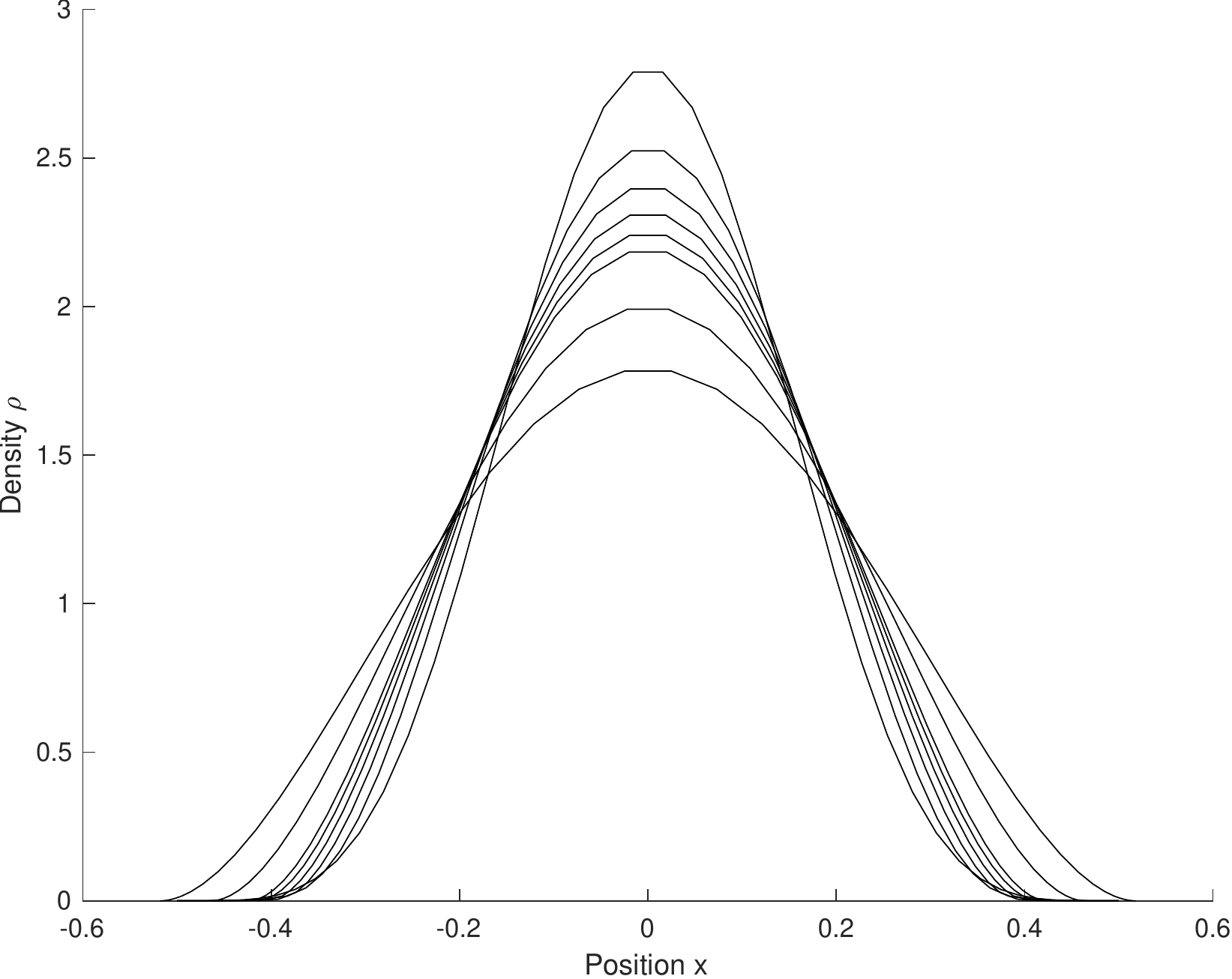}
  \newline
  \includegraphics[width=0.24\textwidth]{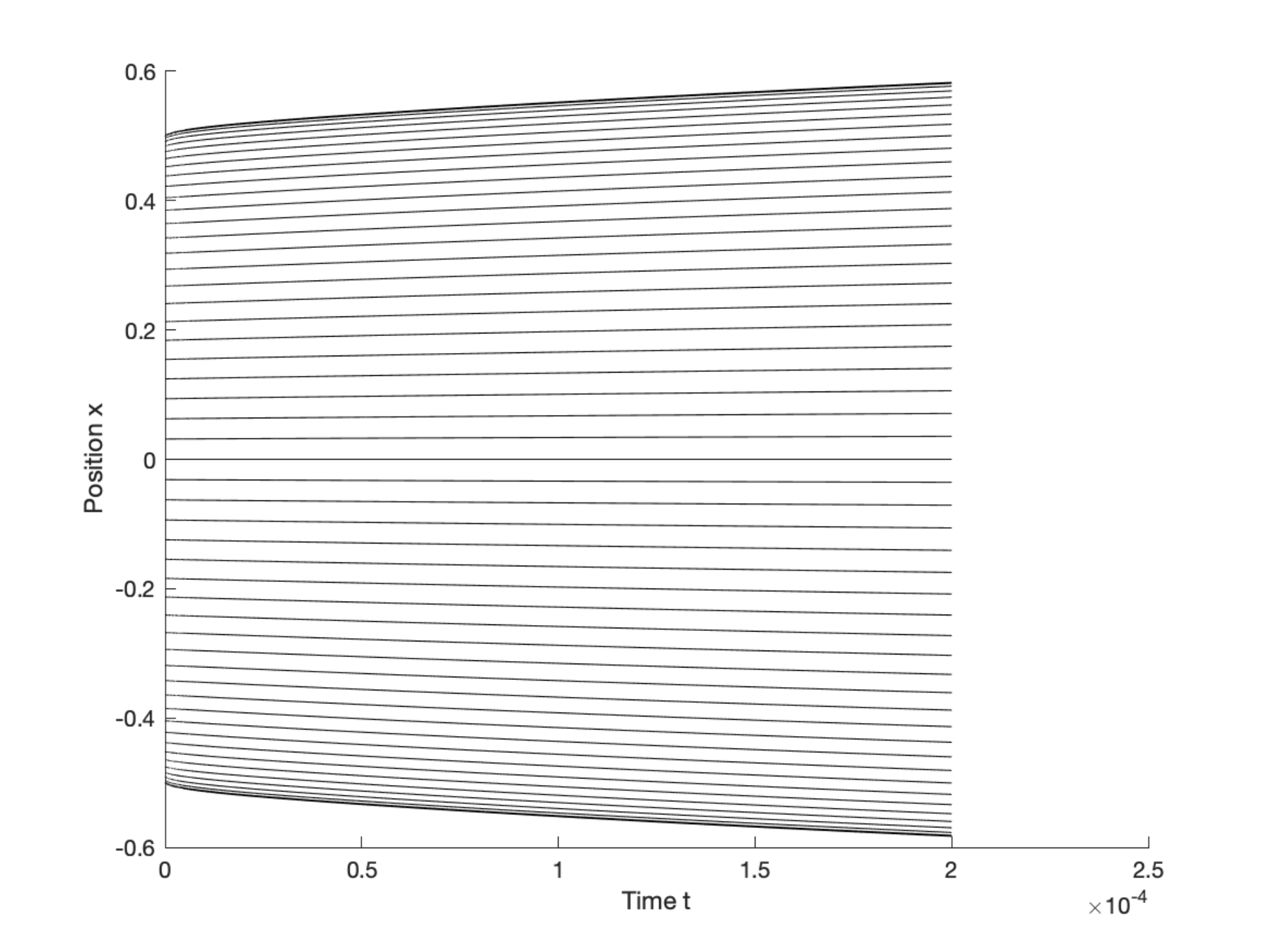}
  \includegraphics[width=0.24\textwidth]{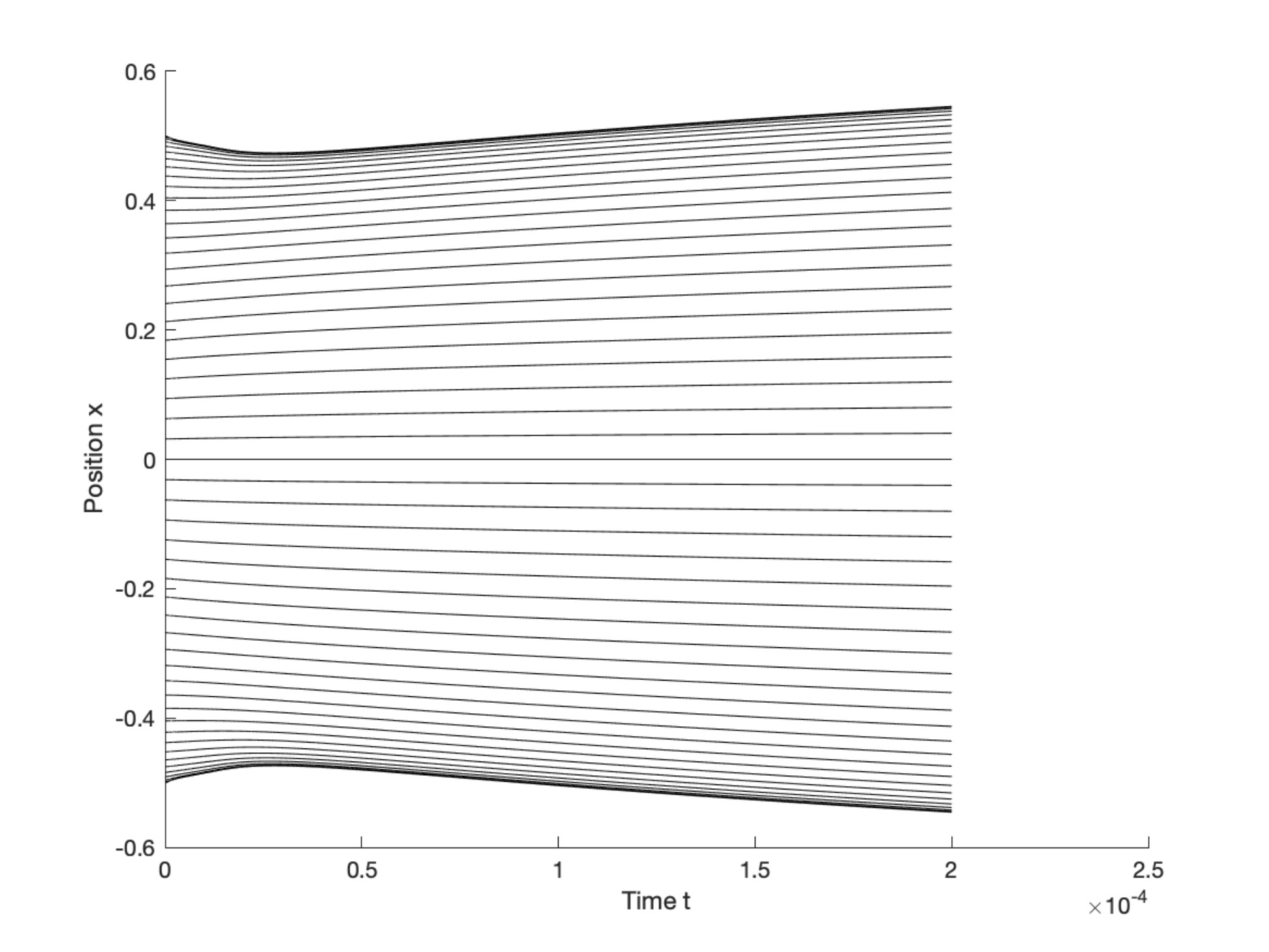}
  \includegraphics[width=0.24\textwidth]{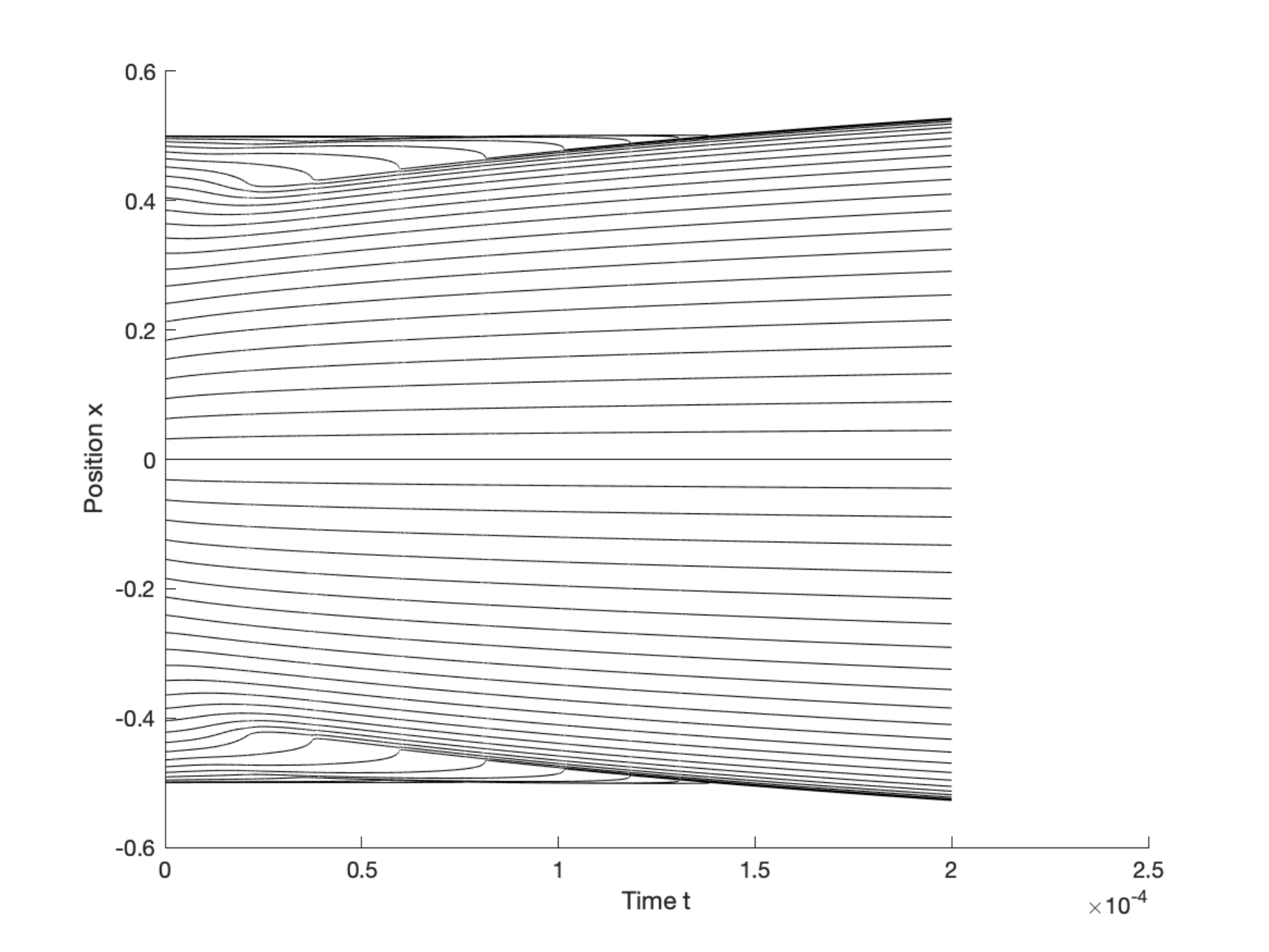}
  \includegraphics[width=0.24\textwidth]{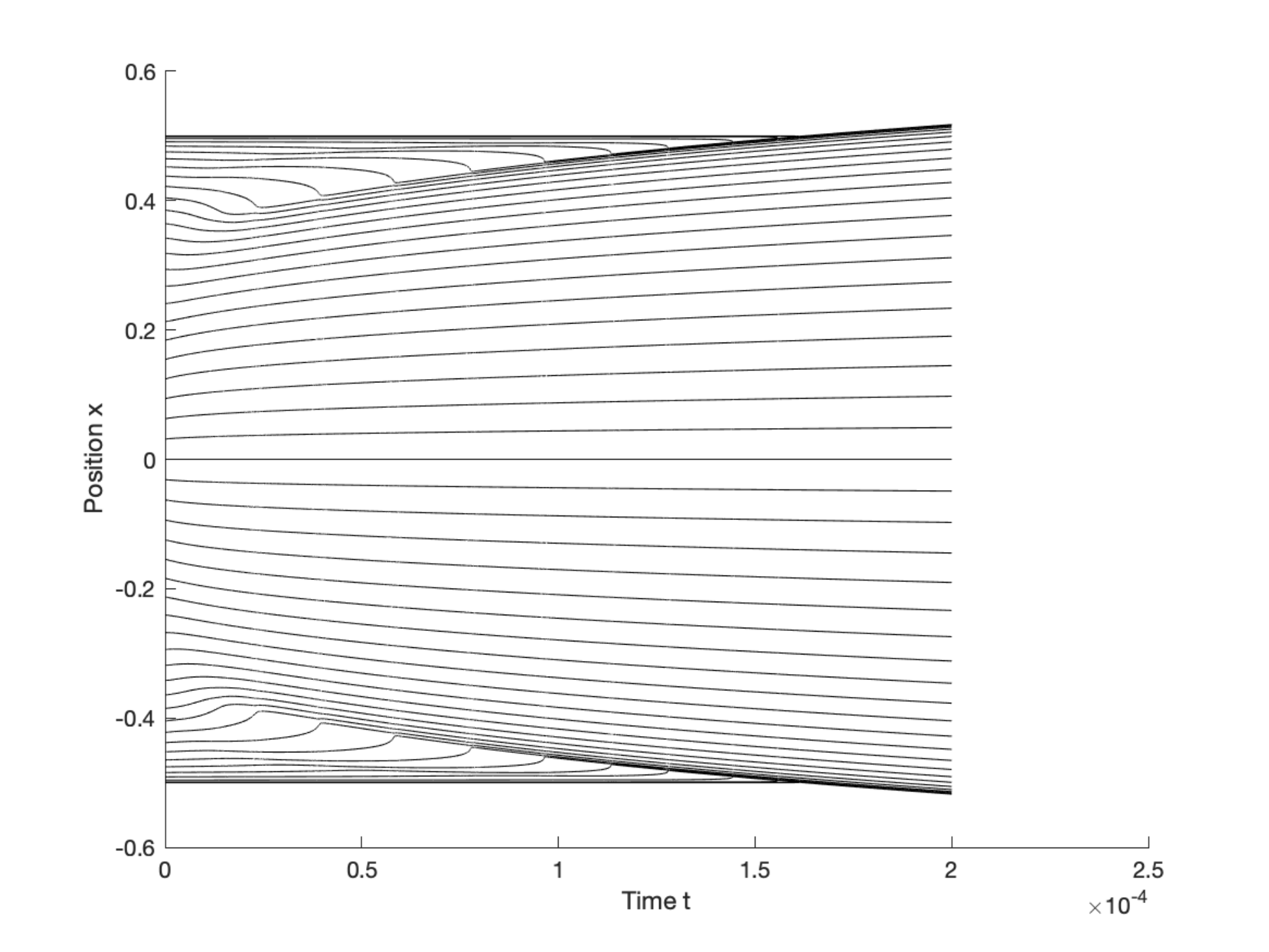}
  \caption{Qualitative illustration of the solutions to the discrete thin film equation for $K=50$:
    for exponents $q=1.5,\,2.5,\,3.5,\,4.5$ from left to right,
    snapshots of the densities $\rho$
    at times $t=0,1,2,3,4,5,10,20\times 10^{-5}$ are shown on top,
    and trajectories of the Lagrangian points directly below.}
  \label{fig:qualTFE}
\end{figure}

For the discrete porous medium equation \eqref{eq:devolPME},
the waiting time phenomenon is nicely illustrated by the trajectories
in the last two plots in the lower row of Figure \ref{fig:qualPME}:
in the beginning, the outermost points remain at their initial position without any visible movement
and then gain momentum quite abruptly.
A more quantitative analysis is difficult since there is no clearly defined distinction
between the occurence of a waiting time and an initially very slow motion of the edge of support
for the spatially discrete solutions.
Still, to make some quantitative statement, we have made an ad hoc definition 
of an approximative measure for the duration of the waiting time:
we use the supremum $T$ of all times $t\ge0$ such that $x_1(t)\ge-\frac12$,
that is, the first time at which the left-most mass package has completely left $\bar u$'s support.
The thus obtained values $T$ are in good agreement with the time
at which the plots of the Lagrangian trajectories suggest the first significant motion of the edges of support.

\begin{figure}
  \centering
  \includegraphics[width=0.4\textwidth]{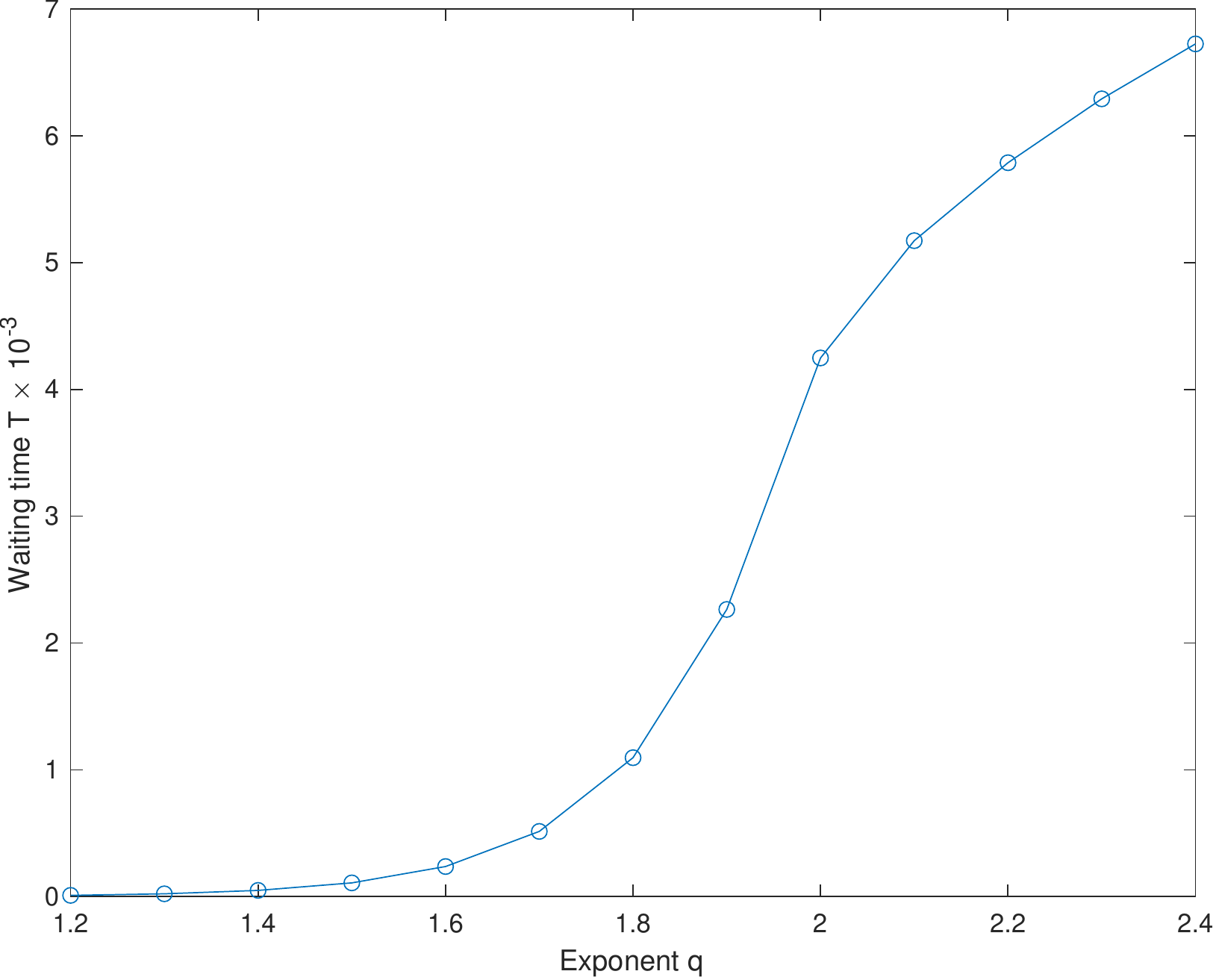}
  \hspace{20pt}
  \includegraphics[width=0.4\textwidth]{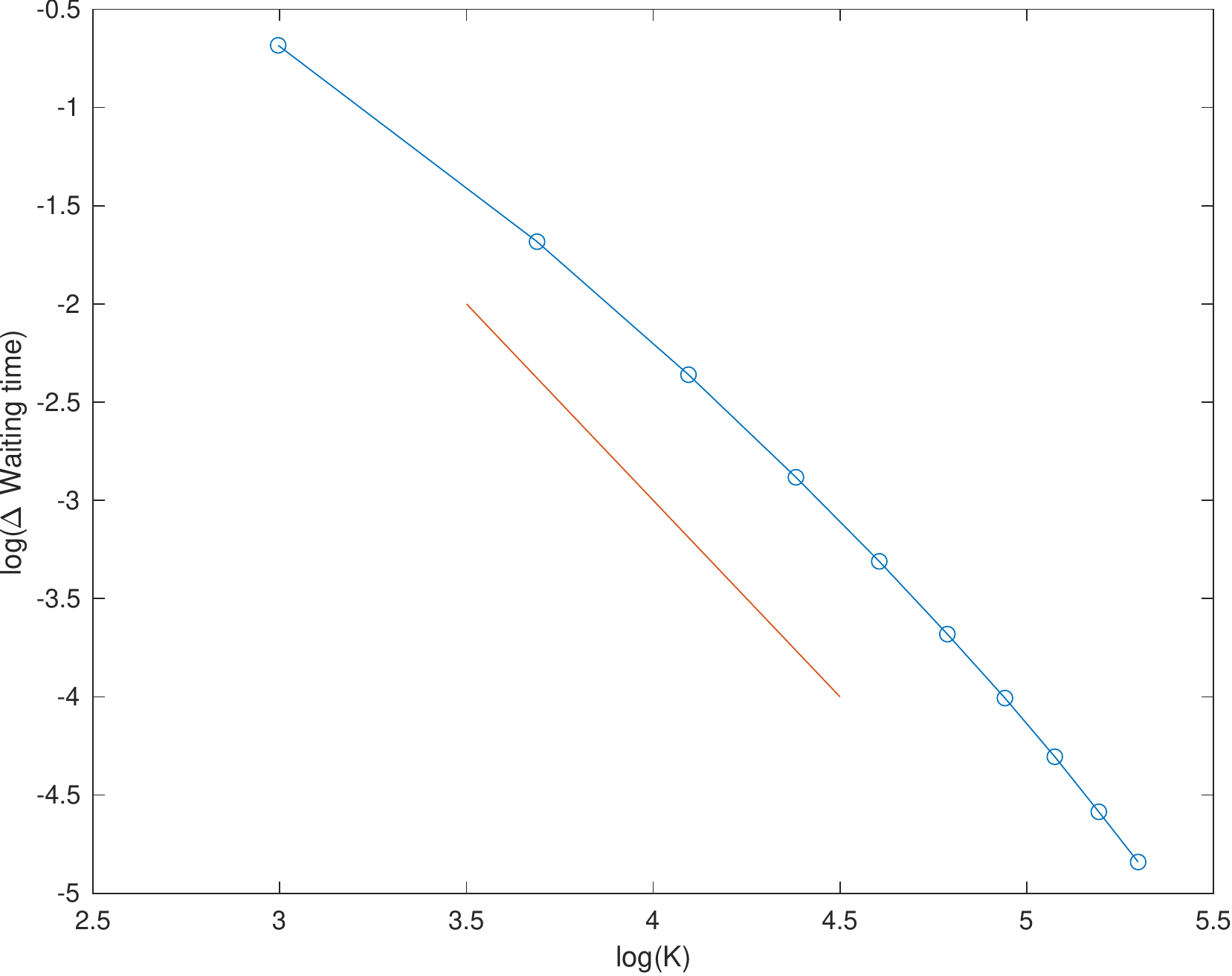}  
  \caption{Values of the estimated waiting times for solutions of the discrete porous medium equation.
    Left: estimated waiting time in simulations with $K=400$ versus exponent $q$ in \eqref{eq:dinit}.
    Right: double logarithmic plot of the deviation of the estimated waiting time for fixed exponent $q=2.2$
    with $K=20,\,40,\ldots,200$ from the reference value at $K=400$;
    the straight line has a slope of $-2$.}
  \label{fig:wtimeplot}
\end{figure}

From reference solutions with $K=400$,
we have computed that approximate waiting time $T$ for different values of $q$ between $1.2$ and $2.4$,
see Figure \ref{fig:wtimeplot} left.
From the theory of the PDE \eqref{eq:PME1},
one would expect no waiting time (i.e., $T=0$) for $q$ below the critical value $q^*=2$,
and then a jump to a positive value at $q=q^*$, followed by a continuous growth of $T$ with $q>q^*$.
Clearly, such a sharp transition cannot be expected after discretization,
at least not for our ad hoc approximation of the waiting time,
for the reasons that have been explained above.
Still, the plot reflects the expected behaviour quite well:
it shows a relatively steep growth of $T$ as $q$ approaches the critical value $q^*=2$ from below,
and once $q$ is above the critical value, $T$ continuous to grow, but at a slower rate.

We have further studied the convergence of the estimated waiting time for solutions with different values of $K$
towards the reference value at $K=400$, see Figure \ref{fig:wtimeplot} right.
The approximation error is of the order $O(K^{-2})$, which is expected:
we have $\bar x_1-\bar x_0=O(K^{-2})$ by construction,
and this is proportional to the time that it takes $x_1(t)$ to reach position $x=-\frac12$
once it has gained speed.
Note that already for $K=20$, the approximation of the waiting time
differs from the reference value by only about twenty percent.

In the case of the discrete thin film equation,
the qualitative behaviour of solutions is too complex to admit a similar quantitative evaluation of the numerical results.
Instead, we only briefly comment on the results reported in Figure \ref{fig:qualTFE}.
There is very obviously no waiting time for $q=1.5$ and for $q=2.5$, respectively:
in the first case, the support spreads immediately after initialization,
in the second, the support recedes, and then expands later.
These observations are in agreement with the expected behaviour for the PDE \eqref{eq:TF1}.
For $q=3.5$ and $q=4.5$, a waiting time is very clearly observed.
For $q=4.5$, this is again in perfect agreement with
the general theory, see \cite{FischerStrongAndVeryWeakSlippage}, and our own Theorem \ref{thm:dtfmain}.
On the contrary, the occurence of a waiting time for $q=3.5$ is rather unexpected,
but does not contradict Theorem \ref{thm:dtfmain} or the available \emph{sufficient} criteria for the PDE \eqref{eq:TF1}.


\appendix

\section{Elementary inequalities}
\begin{lemma}
  For any positive real numbers $x\neq y$, and all $p\in(0,1)$,
  \begin{align}
    \label{eq:diffquot}
    0\le\frac{x^{1+p}-y^{1+p}}{x-y}\le\frac{1+p}{2^p}(x+y)^p.
  \end{align}
\end{lemma}
\begin{proof}
  The estimate from below follows by monotonicity of $t\mapsto t^{1+p}$.
  The bound from above can be derived via Taylor expansion as follows:
  let $m:=\frac{x+y}2$,
  then
  \begin{align*}
    x^{1+p} &= m^{1+p}+(1+p)m^p(x-m)+\frac{p(1+p)}2m^{-(1-p)}(x-m)^2-\frac{p(1-p^2)}6\xi^{-(2-p)}(x-m)^3, \\
    y^{1+p} &= m^{1+p}+(1+p)m^p(y-m)+\frac{p(1+p)}2m^{-(1-p)}(y-m)^2-\frac{p(1-p^2)}6\eta^{-(2-p)}(y-m)^3,
  \end{align*}
  where $\xi\in(0,x)$ and $\eta\in(0,y)$ are suitable intermediate values.
  We subtract the second equation from the first, 
  and divide by $x-y=2(x-m)=-2(y-m)$:
  \begin{align*}
    \frac{x^{1+p}-y^{1+p}}{x-y} = (1+p)m^p - \frac{p(1-p^2)}{48}\big(\xi^{-(2-p)}+\eta^{-(2-p)}\big)(x-y)^2 \le (1+p)m^p.
  \end{align*}
\end{proof}
\begin{lemma}
  For any positive real numbers $x,y$, and all $p\in(0,1)$,
  \begin{align}
    \label{eq:diffquot2}
    \big|x^{1+p}-y^{1+p}-(1+p)y^p(x-y)\big|\le p y^{-1+p}(x-y)^2.
  \end{align}
\end{lemma}
\begin{proof}
  After division by $y^{1+p}>0$, the estimate \eqref{eq:diffquot2} becomes
  \begin{align}
    \label{eq:diffquot2b}
    \big|z^{1+p}-1-(1+p)(z-1)\big|\le p(z-1)^2,
  \end{align}
  where $z:=x/y>0$.
  If $z\ge1$, then \eqref{eq:diffquot2b} is directly obtained
  by means of a Taylor expansion of $z\mapsto z^{1+p}$ around $\bar z=1$,
  that is
  \begin{align*}
    z^{1+p} = 1 + (1+p)(z-1) + \frac12 (1+p)p\zeta^{-1+p}(z-1)^2,
  \end{align*}
  where $\zeta\ge1$ is an intermediate value between $\bar z=1$ and $z$.
  Indeed, it suffices to observe that $0\le \frac{1+p}{2}\zeta^{-1+p}\le1$ since $p\in(0,1)$,
  and \eqref{eq:diffquot2b} follows.
  If instead $0\le z\le1$, consider
  \begin{align}
    \label{eq:dummy001}
    z^{1+p}-1-(1+p)(z-1) = (1+p) \int_z^1(1-\zeta^p)\dd\zeta
    = p(1+p) \int_z^1\left[\int_\zeta^1\eta^{p-1}\dd\eta\right]\dd\zeta.
  \end{align}
  We re-write the double integral,
  performing a change of variables $\zeta=1-(1-z)t$, $\eta=1-(1-z)s$,
  \begin{align*}
    J(z):=\int_z^1\left[\int_z^1\mathbf{1}_{\zeta\le\eta}\eta^{p-1}\dd\eta\right]\dd\zeta
    = (1-z)^2\int_0^1\left[\int_0^1\mathbf{1}_{s\le t}\big(1-(1-z)s\big)^{p-1}\dd s\right]\dd t.
  \end{align*}
  Since $p<1$, the integrand is monotonically decreasing with respect to $z\in[0,1]$.
  Therefore, $(1-z)^{-2}J(z)$ is decreasing, and so $ J(z) \le (1-z)^2J(0)$,
  which means in view of \eqref{eq:dummy001} that
  \begin{align*}
    z^{1+p}-1-(1+p)(z-1) = (1+p) \int_z^1(1-\zeta^p)\dd\zeta \leq p(1+p)J(0)(1-z)^2 = p(1-z)^2.
  \end{align*}
\end{proof}
\begin{lemma}
  For any positive grid function $f$, and all $p\in(0,1)$,
  \begin{align}
    \label{eq:laplace1}
    \big|\Delta_\kappa(f^{1+p})-(1+p)f_\kappa^p\Delta_\kappa f\big|
    \le pf_\kappa^{p-1}\big[\big(\dff_{\kapph}f\big)^2+\big(\dff_{\kapmh}f\big)^2\big]
  \end{align}
\end{lemma}
\begin{proof}
  This is an immediate application of \eqref{eq:diffquot2}:
  \begin{align*}
    \big|\Delta_\kappa(f^{1+p})-(1+p)f_\kappa^p\Delta_\kappa f\big|
    &=\cell^{-2}\big|\big(f_{\kappa+1}^{1+p}-2f_\kappa^{1+p}+f_{\kappa-1}^{1+p}\big)
      -(1+p)f_\kappa^p(f_{\kappa+1}-2f_\kappa+f_{\kappa-1})\big| \\
    &\le\cell^{-2}\big|\big(f_{\kappa+1}^{1+p}-f_\kappa^{1+p}\big)-(1+p)f_\kappa^p(f_{\kappa+1}-f_\kappa)\big| \\
    &\qquad+\cell^{-2}\big|\big(f_{\kappa-1}^{1+p}-f_\kappa^{1+p}\big)-(1+p)f_\kappa^p(f_{\kappa-1}-f_\kappa)\big| \\
    &\le p \cell^{-2}f_\kappa^{p-1}\big[(f_{\kappa+1}-f_\kappa)^2+(f_{\kappa-1}-f_\kappa)^2\big] \\
    &= p f_\kappa^{p-1}\big[(\dff_{\kapph}f)^2+(\dff_\kapmh f)^2\big].
  \end{align*}
\end{proof}
\begin{lemma}
  For any positive real numbers $x,y$, and all $p\in(0,1)$,
  \begin{align}
    \label{eq:moreelementary}
    \big(x^{1+p}-y^{1+p}\big)(x-y)\le(x^{p}+y^{p})(x-y)^2.
  \end{align}  
\end{lemma}
\begin{proof}
  Without loss of generality, assume that $x>y$.
  For \eqref{eq:moreelementary}, we need to show that
  \begin{align*}
    x^{1+p}-y^{1+p}\le(x^{p}+y^{p})(x-y) = x^{1+p}-y^{1+p}+xy^{p}-x^{p}y. 
  \end{align*}
  This is obviously true since $0\le (xy)^p(x^{1-p}-y^{1-p})$.
\end{proof}
\begin{lemma}
  For any positive real numbers $x,y$, and all $p\in(0,1)$,
  \begin{align}
    \label{eq:evenmoreelementary}
    \big(x^{2+p}-y^{2+p}\big)(x-y)\ge(x^{1+p}+y^{1+p})(x-y)^2.
  \end{align}
  Moreover, if $a,b$ are non-negative weights, then
  \begin{align}
    \label{eq:evenmoreelementary2}
    (ax^{2+p}-by^{2+p})(x-y)\ge(ax^{1+p}+by^{1+p})(x-y)^2 - |a-b|(x^{2+p}+y^{2+p})|x-y|.
  \end{align}
\end{lemma}
\begin{proof}
  Without loss of generality, assume that $x>y$.
  For \eqref{eq:evenmoreelementary}, we need to show that
  \begin{align*}
    x^{2+p}-y^{2+p}\ge(x^{1+p}+y^{1+p})(x-y) = x^{2+p}-y^{2+p}+xy^{1+p}-x^{1+p}y. 
  \end{align*}
  This is obviously true since $0\ge xy(y^p-x^p)$.
  For the proof of \eqref{eq:evenmoreelementary2},
  we use \eqref{eq:evenmoreelementary} as follows:
  \begin{align*}
    &(ax^{2+p}-by^{2+p})(x-y)
    \ge \frac{a+b}2(x^{2+p}-y^{2+p})(x-y) - \frac{|a-b|}2(x^{2+p}+y^{2+p})|x-y| \\
    & \ge \frac{a+b}2(x^{1+p}+y^{1+p})(x-y)^2 - \frac{|a-b|}2(x^{2+p}+y^{2+p})|x-y| \\
    &\ge (ax^{1+p}+by^{1+p})(x-y)^2- \frac{|a-b|}2(x^{2+p}+y^{2+p}+(x^{1+p}+y^{1+p})|x-y|)|x-y|.
  \end{align*}
  From here, \eqref{eq:evenmoreelementary2} follows immediately since
  \begin{align*}
    (x^{1+p}+y^{1+p})|x-y|\le
    x^{2+p}+y^{2+p}
  \end{align*}
  which is obvious for $x=y$ and follows for $x>y$ by differentiation with respect to $x$.
\end{proof}

\begin{proof}[Proof of Lemma~\ref{LemmaGNS}]
  We concentrate on the proof of \eqref{eq:dGNS},
  and discuss the necessary changes for \eqref{eq:dGNS4} afterwards.
  A first intermediate result is
  \begin{align}
    \label{eq:technical3}
    \max_{\hi\le\kappa\le k^*-\hi}f_\kappa^{2(1+r)}
    \le A_{2,r}^{\frac{1+r}{1-r}}\left(\sum_{k=0}^{k^*-1}(\dff_kf)^2\cell_k\right)\left(\sum_{\kappa=\hi}^{k^*-\hi}f_\kappa^{2r}\cell_\kappa\right).    
  \end{align}
  Indeed, thanks to the elementary inequality \eqref{eq:diffquot} from the appendix,
  and recalling that $f_{-\hi}=0$,
  we have that
  \begin{align}
    \label{eq:gnsaux}
    \max_{\hi\le\kappa\le k^*-\hi}f_\kappa^{1+r}
    \le \sum_{k=0}^{k^*-1}\big|f_\kph^{1+r}-f_\kmh^{1+r}\big|
    \le 2^{-r}(1+r)\sum_{k=0}^{k^*-1}(f_\kph^r+f_\kmh^r)|\dff_k f|\cell_k.
  \end{align}
  Next, we apply the Cauchy-Schwarz inequality to the sum,
  and use \eqref{eq:meshestimate}:
  \begin{align*}
    \sum_{k=0}^{k^*-1}(f_\kph^r+f_\kmh^r)\big|\dff_kf\big|\cell_k
    &\le \left(2\sum_{k=0}^{k^*-1}(f_\kph^{2r}+f_\kmh^{2r})\cell_k\right)^{\frac12}
    \left(\sum_{k=0}^{k^*-1}(\dff_kf)^2\cell_k\right)^{\frac12} \\
    &\le \left(2(1+\mratio)\sum_{\kappa=\hi}^{{k^*}-\hi}f_\kappa^{2r}\cell_\kappa\right)^{\frac12}
    \left(\sum_{k=0}^{k^*-1}(\dff_kf)^2\cell_k\right)^{\frac12} .
  \end{align*}
  After taking the square, we arrive at \eqref{eq:technical3}.
  From here, \eqref{eq:dGNS} is obtained as follows:
  \begin{align*}
    \sum_{\kappa=\hi}^{k^*-\hi}f_\kappa^2\cell_\kappa
    &\le \left(\max_{\hi\le\kappa\le k^*-\hi}f_\kappa^{2(1+r)}\right)^{\frac{1-r}{1+r}}\sum_{\kappa=\hi}^{k^*-\hi}f_\kappa^{2r}\cell_\kappa \\
    &\le A_{2,r} \left[\left(\sum_{k=0}^{k^*-1}(\dff_kf)^2\cell_k\right)\left(\sum_{\kappa=\hi}^{k^*-\hi}f_\kappa^{2r}\cell_\kappa\right)\right]^{\frac{1-r}{1+r}}
      \sum_{\kappa=\hi}^{k^*-\hi}f_\kappa^{2r}\cell_\kappa.
  \end{align*}
  The argument for \eqref{eq:dGNS4} follows the same lines:
  consider \eqref{eq:gnsaux} with $r:=3s$, and instead of the Cauchy-Schwarz inequality,
  apply H\"older's inequality with exponents $4$ and $4/3$ to the sums, and take the fourth power.
  This produces the following analogue of \eqref{eq:technical3}:
  \begin{align*}
    \max_{\hi\le\kappa\le k^*-\hi}f_\kappa^{4(1+3s)}
    \le A_{4,s}^{\frac{1+3s}{1-s}}\left(\sum_{k=0}^{k^*-1}(\dff_kf)^4\cell_k\right)\left(\sum_{\kappa=\hi}^{k^*-\hi}f_\kappa^{4s}\cell_k\right)^3.
  \end{align*}
  Similarly as before,
  \begin{align*}
    \sum_{\kappa=\hi}^{k^*-\hi}f_\kappa^4\cell_\kappa
    &\le \left(\max_{\hi\le\kappa\le k^*-\hi}f_\kappa^{4(1+3s)}\right)^{\frac{1-s}{1+3s}}\sum_{\kappa=\hi}^{k^*-\hi}f_\kappa^{4s}\cell_\kappa \\
    &\le A_{4,s} \left[\left(\sum_{k=0}^{k^*-1}(\dff_kf)^4\cell_k\right)\left(\sum_{\kappa=\hi}^{k^*-\hi}f_\kappa^{4s}\cell_\kappa\right)^3\right]^{\frac{1-s}{1+3s}}
      \sum_{\kappa=\hi}^{k^*-\hi}f_\kappa^{4s}\cell_\kappa,
  \end{align*}
  which is \eqref{eq:dGNS4}.
\end{proof}

\bibliographystyle{abbrv}
\bibliography{thinfilm}

\end{document}